\documentclass[reqno]{amsart}
\usepackage{color}
\usepackage{amssymb,amsmath,amsthm,amstext,amsfonts}
\usepackage[dvips]{graphicx}
\usepackage{psfrag}
\usepackage{url}

\makeatletter \@addtoreset{equation}{section} \makeatother

\renewcommand\thetable{\thesection.\@arabic\c@table}

\theoremstyle{plain}
\newtheorem{maintheorem}{Theorem}

\newtheorem{maincorollary}{Corollary}

\newtheorem{mainproposition}{Proposition}

\newtheorem{theorem}{Theorem}[section]
\newtheorem{proposition}{Proposition}[section]
\newtheorem{lemma}{Lemma}[section]
\newtheorem{corollary}{Corollary}[section]
\newtheorem{definition}{Definition}[section]

\newtheorem{example}{Example}[section]

\newcommand{\htop}{h_{\topp}}

\newcommand{\al} {\alpha}       
        
\newcommand{\ga} {\gamma}    
\newcommand{\de} {\delta}       
\newcommand{\vep}{\varepsilon}

\newcommand{\la} {\lambda}      \newcommand{\La}{\Lambda}

\newcommand{\si} {\sigma}       \newcommand{\Si}{\Sigma}

\newcommand{\R}{\mathbb{R}}

\newcommand{\diam}{\operatorname{diam}}

\newcommand{\topp}{\operatorname{top}}

\newcommand{\Leb}{Leb}
\newcommand{\ti}{\tilde }

\newcommand{\cP}{\mathcal{P}}

\newcommand{\cE}{\mathcal{E}}
\newcommand{\cM}{\mathcal{M}}

\newcommand{\cG}{\mathcal{G}}

\newcommand{\cQ}{\mathcal{Q}}
\newcommand{\cF}{\mathcal{F}}

\newcommand{\var}{\text{var}}

\begin{document}

\title{Weak specification properties and large deviations for non-additive potentials}

\author{Paulo Varandas}
\address{Departamento de Matem\'atica, Universidade Federal da Bahia\\
  Av. Ademar de Barros s/n, 40170-110 Salvador, Brazil.}
\email{paulo.varandas@ufba.br \\ \url{http://www.pgmat.ufba.br/varandas/}}

\author{Yun Zhao}
\address{Departament of Mathematics, Soochow University\\
  Suzhou 215006, Jiangsu, P.R. China}
\email{zhaoyun@suda.edu.cn}

\date{\today}

\begin{abstract}
We  obtain large deviation bounds for the measure of deviation
sets associated to asymptotically additive and sub-additive
potentials under some weak specification properties. In particular
a large deviation principle is obtained in the case of uniformly
hyperbolic dynamical systems. Some applications to the study
of the convergence of Lyapunov exponents are given.
\end{abstract}



\maketitle

\section{Introduction }

The purpose of the theory of large deviations is to study the
rates of convergence of sequences of random variables to some
limit distribution. Some applications of these ideas into the realm of Dynamical Systems
have been particularly useful to estimate the velocity at which
time averages of typical points of ergodic invariant measures
converge to the space average as guaranteed by Birkhoff's ergodic
theorem. More precisely, given a continuous transformation $f$ on
a compact metric space $M$ and a reference measure $\nu$, one
interesting question is to obtain sharp estimates for the
$\nu$-measure of the deviation sets $ \{
    x \in M : \frac{1}{n} \sum_{j=0}^{n-1} g(f^j(x)) > c
\} $ for all continuous functions $g: M \to \R$ and real numbers
$c$. We refer the reader to \cite{You90,Ki90,KN91,EKW94, PS05,
AP06,MN08,RY08, CoRi08, Yu07, Mel09, PS09, Co09, Chu11,Va12} and the
references therein for an account on recent large deviations
results.

Since many relevant quantities in dynamical systems arise from non-additive sequences,
e.g. the largest Lyapunov exponent for higher dimensional dynamical systems,  Kingman's ergodic theorem
becomes in many situations crucial to study the deviation sets
$
\{ x \in M : \varphi_n(x) > c n \}
$
with respect to some not necessarily additive sequence $\Phi=\{\varphi_n\}_n$ of continuous functions.
Inspired by the pioneering work of Young~\cite{You90} our purpose
in this direction is to provide sharp large deviations estimates
for a wide class of non-additive sequences of continuous
potentials. Our approach uses ideas from the non-additive
thermodynamical formalism and we estimate the measure of deviation
sets in the case that the reference measure satisfies a weak Gibbs
property. Some recent results considering the thermodynamical
formalism of almost additive or sub-additive sequences of
potentials include \cite{FL02, Ba06,Mu06, IY11, FK12}: in all
cases the authors proved that there exists a unique equilibrium
state $\mu_\Phi$ and it is absolutely continuous with respect to a
Gibbs measure $\nu_\Phi$ with density bounded away from zero and
infinity. Building over~\cite{ba96}, M\'eson and
Vericat~\cite{MV09} obtained bounds for large deviations processes
for a family of non-additive potentials $\Phi=\{\varphi_n\}$,
namely those such that $\varphi_n-\varphi_{n-1}\circ f$ converge
uniformly. Our purpose here is to extend the theory beyond the
uniformly hyperbolic context and to consider a broad class of
sub-additive, almost additive and asymptotically additive
sequences of potentials. For simplicity, we will refer to the
previous classes of potentials as non-additive sequences of
potentials.

Let us mention that, since we consider some non-uniformly hyperbolic
dynamical systems or dynamical systems that admit mistakes these in general do not satisfy
the usual specification property. In fact, if on the one hand the specification property holds robustly in the uniformly
hyperbolic setting~\cite{SSY09}, on the other hand dynamical systems with the specification property are rare even
among partially hyperbolic dynamical systems in dimension three~\cite{SVY}. For that reason we assume the
dynamical system to satisfy some weaker specification property.
In fact, as a physical process evolves it is natural for the
evolving process to change or produce some errors in the
evaluation of orbits. However, a self-adaptable system should
decrease errors over time. This is a  motivation for our study of
the large deviations for non-additive potentials when the systems
admits mistakes. In fact, 
%
not only the notions of specification and topologically mixing
coincide for every one-dimensional \emph{continuous} mapping (see
\cite{Bl83}) as weaker specification properties hold in the
presence of nonuniform hyperbolicity.

Roughly, one proves that the measure of the set of points whose
sequences of values remain far from the space average with respect
to the equilibrium measure decrease exponentially fast. In
particular we obtain a large deviations principle for non-additive
sequences in the uniformly expanding setting.
As important applications we estimate the rate of convergence of the
maximal Lyapunov exponent for some open families of linear cocycles and
and non-conformal expanding maps. We refer the reader to Section 4 for precise
statements and details.

The remainder of this paper is organized as follows. In section 2,
we present some definitions and fundamental notions necessary to
state our results. In Section 3 we present statements of the main
results in this paper. Some examples are given in Section 4 while
the proofs of the main results are given in section 5. Finally in
the Appendix~A we estimate the measure of mistake dynamical balls
in the uniformly expanding setting.

\section{Preliminaries}\label{s.preliminaries}

Throughout this paper,  $(M,f)$ denotes a continuous
 dynamical systems in the sense that $f:M\rightarrow M$ is a (piecewise) continuous
 transformation on the compact metric space $M$ with a metric $d $.  Invariant Borel
probability measures are associated with $(M,f)$. Let
$\mathcal{M}_f$ and $\mathcal{E}_f$  denote  the space of
$f$-invariant Borel probability measures and the set of
$f$-invariant ergodic Borel probability measures, respectively.

\subsection{Specification properties }\label{subsec.specification}

Specification properties are very useful to obtain existence of equilibrium states as well as
large deviation principles. Here we introduce and discuss some different notions.

\begin{definition}
\label{d.strong.specification}
We say that a map $f$ satisfies the \emph{specification property} if
for any $\vep>0$ there exists an integer $N=N(\vep)\geq 1$ such that
the following holds: for every $k\geq 1$, any points $x_1,\dots,
x_k$, and any sequence of positive integers $n_1, \dots, n_k$ and
$p_1, \dots, p_k$ with $p_i \geq N(\vep)$
there exists a point $x$ in $M$ such that
$$
\begin{array}{cc}
d\Big(f^j(x),f^j(x_1)\Big) \leq \vep, &\forall \,0\leq j \leq n_1
\end{array}
$$
and
$$
\begin{array}{cc}
d\Big(f^{j+n_1+p_1+\dots +n_{i-1}+p_{i-1}}(x) \;,\; f^j(x_i)\Big)
        \leq \vep &
\end{array}
$$
for every $2\leq i\leq k$ and $0\leq j\leq n_i$.
\end{definition}

The previous notion is  slightly weaker than the one introduced by Bowen~\cite{Bo71},
that requires that any finite sequence of pieces of orbit is well approximated by periodic orbits.
Although robust specification property for diffeomorphisms is satisfied only by uniformly hyperbolic
dynamical systems (see Sakai, Sumi and Yamamoto~\cite{SSY09}) we know by Blokh~\cite{Bl83}
that the notions of specification and topologically mixing coincide for
every one-dimensional \emph{continuous} mapping.
This is no longer true if the one-dimensional map fails to be continuous (see e.g.~\cite{Buz97}).

To define other weak form of specification we first recall the
definitions of mistake function and mistake dynamical balls  which
are due to Thompson \cite{Th10}, Pfister and Sullivan \cite{ps}.
Given $\vep_0>0$, the function $g:\mathbb{N}\times
(0,\vep_0]\rightarrow \mathbb{N}$ is called a \emph{mistake
function} if for all $\vep\in (0,\vep_0]$ and all $n\in
\mathbb{N}$, $g(n,\vep)\leq g(n+1,\vep)$ and $\lim_{n}
g(n,\vep)/n=0$. By a slight abuse of notation, we set
$g(n,\vep)=g(n,\vep_0)$ for every $\vep>\vep_0$.
Moreover, for any subset of integers $\Lambda\subset [0,N]$, we
will use the family of distances in the metric space $M$ given by
$d_\Lambda(x,y)=\max \{d(f^ix,f^iy):i\in\Lambda\}$ and consider
the balls $B_{\Lambda}(x,\vep)=\{y\in M:d_\Lambda(x,y)<\vep\}$.
Hence we can now consider mistake dynamical balls.
Given a mistake function $g$, $\vep>0$ and $n\ge 1$, the
\emph{$(n,\varepsilon)-$mistake dynamical ball} $B_n(g;x,\vep)$ of
radius $\vep$ and length $n$ associated to $g$ is defined by
\begin{eqnarray*}
B_n(g;x,\vep)&=&\{ y\in M\mid y\in B_{\Lambda}(x,\vep)~\hbox{for
some}~\Lambda\in I(g;n,\vep)\}\\
&=&\bigcup_{\Lambda\in I(g;n,\vep)}B_{\Lambda}(x,\vep)
\end{eqnarray*}
where $I(g;n,\vep)=\{ \Lambda\subset [0,n-1]\cap\mathbb{N}\mid \#
\Lambda \geq n-g(n,\vep)\}$. A set $F\subset Z$ is
$(g;n,\varepsilon)-$separated for $Z$ if for every $x,y\in F$ with
$x\neq y$ implies $d_\Lambda(x,y)>\varepsilon,\forall \Lambda\in
I(g;n,\varepsilon) $. The dual definition is as follows.  A set
$E\subset Z$ is $(g;n,\varepsilon)-$spanning for $Z$ if for all
$z\in Z$, there exists $x\in E$ and $\Lambda\in I(g;n,\varepsilon)
$ such that $d_\Lambda(x,z)\leq \varepsilon$.

\begin{definition}\label{d.almost.specification}
Let $g$ be a mistake function. We say that $f$ satisfies the
\emph{$g$-almost specification property} if there exists $\vep>0$
and a positive integer $N(g,\vep)$ such that the following
property holds: for every $k\geq 1$, any points $x_1,\dots, x_k$,
and any positive integers $n_1, \dots, n_k$ with $n_i \geq
N(g,\vep)$ it follows that
$$
\bigcap_{i=1}^{k} f^{-\sum_{j=0}^{i-1} n_j} (B_{n_i}(g; x_i,\vep))
    \neq \emptyset
$$
where $n_0=0$.
\end{definition}

The later property holds for all $\beta$-transformations (see \cite{ps,Th10}). In fact
Thompson~\cite{Th10} introduced a more general property  where the value $\vep$ is replaced by
several values $\vep_1,\dots, \vep_k$. However this weaker requirement is suitable for our purposes.

\subsection{Non-additive potentials}

 Let $C(M)$ denote the space of continuous
functions from $M$ to $\mathbb R$. A sequence
$\Phi=\{\varphi_n\}\subset C(M) $  is a \emph{sub-additive}
(respectively superadditive) sequence of potentials if
$\varphi_{m+n}\le \varphi_{m}+\varphi_{n}\circ f^m$ (respectively
$\varphi_{m+n}\ge \varphi_{m}+\varphi_{n}\circ f^m$) for every
$m,n\ge 1$.

We say that the sequence  $\Phi=\{\varphi_n\}\subset C(M)$ is an
\emph{almost additive} sequence of potentials, if there exists a
uniform constant $C>0$ such that $\varphi_{m}+\varphi_{n}\circ
f^m-C \le \varphi_{m+n}\le \varphi_{m}+\varphi_{n}\circ f^m+C$ for
every $m,n\ge 1$. Finally, we say that $\Phi=\{\varphi_n\}\subset
C(M)$ is an \emph{asymptotically additive} sequence of potentials,
if for any $\xi>0$ there exists a continuous function
$\varphi_\xi$ such that
\begin{equation}\label{eq.asymptotic}
\limsup_{n\to\infty} \frac1n \left\| \varphi_n - S_n \varphi_\xi
\right\| <\xi
\end{equation}
where $S_n\varphi_\xi=\sum_{j=0}^{n-1}\varphi_\xi\circ f^j$
denotes the usual Birkhoff sum, and $||\cdot||$ is the sup norm on
the Banach space $C(M)$. Let
$\mathcal A$ denote the set of asymptotically additive potentials.
The following result establishes the relation between these notions.

\begin{proposition}\label{prop:relations}
The following properties hold:
\begin{enumerate}
\item If $\Phi=\{\varphi_n\}$ is almost additive  there exists
$C>0$ such that the sequence
    $\Phi_C=\{\varphi_n+C\}$ is sub-additive and $\Phi_{-C}=\{\varphi_n-C\}$ is superadditive;
\item If $\Phi=\{\varphi_n\}$ is almost additive then it is
asymptotically additive, and for any  $\xi>0$
    there exists $k=k(\xi)\ge 1$ so that
    $
    \limsup_{n\to\infty} \frac1n \big\| \varphi_n - S_n \big( \frac1k \varphi_k \big)
    \big\|<\xi.
    $
\end{enumerate}
\end{proposition}

\begin{proof}
Part (1) is obvious from the definitions. Part (2) is contained in
Proposition A.5 of \cite{FH10} or Proposition 2.1 of \cite{ZZC11}.
\end{proof}

 Example \ref{bg:almostadditive}  provides examples of sub-additive and
superadditive potentials that are almost additive, while in
Example \ref{acr} we exhibit some sub-additive and superadditive
potentials that are asymptotically additive.
By Kingman's subadditive ergodic theorem it follows that for every
sub-additive potential $\Phi=\{\varphi_n\}$ and every
$f$-invariant ergodic probability measure $\mu$ it holds
\begin{equation}\label{eq.Kingman}
\lim_{n\to\infty} \frac1n \varphi_n(x)
    =\inf_{n\geq 1} \frac1n \int \varphi_n \;d\mu=: \cF_*(\mu,\Phi),
    \quad \text{ for $\mu$-a.e. $x$}.
\end{equation}
In fact, Feng and Huang~\cite{FH10} proved that the same property
holds for asymptotically additive potentials and, consequently,
for all almost additive ones. In that paper, Feng and Huang also
proved that the map $\mu\mapsto \cF_*(\mu,\Phi)$ is continuous
(respectively upper semi-continuous) if $\Phi$ is asymptotically
additive (respectively sub-additive).

\subsection{Non-additive topological pressure and equilibrium states}

We recall the notions of non-additive topological
pressure.  For any $\vep>0$ and $n\in\mathbb N$ we consider the
$(n,\varepsilon)-$dynamical balls $B_n(x,\vep):=\{y\in M: d_n(x,
y)<\vep \}$, where $d_n(x,y)=\max_{0\leq i<n}d(f^ix, f^iy)$.
We say that a set $E\subset M$ is $(n,\vep)$-\emph{separated} if
all distinct $x,y\in E$ satisfy $y\notin B_n(x,\vep)$. If
$\Phi=\{\varphi_n\}$ is a non-additive (namely sub-additive,
almost additive or asymptotically additive) family of potentials, the
\emph{topological pressure of $f$ with respect to $\Phi$}  is
defined by
\[
P(f,\Phi)=\lim_{\vep\rightarrow 0}\limsup_{n\rightarrow
\infty}\frac{1}{n}\log \sup_{E_n}\{Z_n(\Phi, E_n,\vep)\}
\]
where $Z_n(\Phi, E_n,\vep)=\sum_{y\in E_n}\exp(\varphi_n(y))$ and
the supremum is taken over all $(n,\vep)$-separated sets. The
following variational principle relates the non-additive pressure
with the natural modifications of the measure-theoretic free
energy. Recall first a very useful formula to compute the
metric entropy due to Katok~\cite[Theorem I.I]{Ka80}:
if  $\eta$ is an $f$-invariant ergodic probability measure
then
\begin{equation}\label{eq.Katok}
h_\eta(f)= \lim_{\vep \to 0} \limsup_{n\to \infty}
        \frac1n \log N(n,\vep,\de)
        = \lim_{\vep \to 0} \liminf_{n\to \infty}
        \frac1n \log N(n,\vep,\de),
\end{equation}
where $N(n,\vep,\delta)$ is the minimum number of
$(n,\vep)$-dynamical balls necessary to cover a set of
$\eta-$measure larger than $\de$. More generally, it was proven in
\cite{Th10}
 that for any mistake function $g$ under the previous assumptions
it also holds  that
\begin{equation}\label{eq.mistKatok}
h_\eta(f)= \lim_{\vep \to 0} \limsup_{n\to \infty}
        \frac1n \log N(g; n,\vep,\de)
        = \lim_{\vep \to 0} \liminf_{n\to \infty}
        \frac1n \log N(g; n,\vep,\de),
\end{equation}
where the term $N(g;n,\vep,\de)$ stands for the minimum number of
$( n,\vep)$-mistake dynamical balls necessary to cover a set of
$\eta-$measure larger than $\delta$. A generalization of formulas
(\ref{eq.Katok}) and (\ref{eq.mistKatok}) to measure theoretic
pressure was given in \cite{czc12} and \cite{chz}.

\begin{theorem}\label{thm:variational.principle} Let $f:M\rightarrow M$ be a continuous map on the
compact metric space $M$, and $\Phi=\{\varphi_n\}$ a non-additive
potential (namely sub-additive, almost additive or
asymptotically additive). Then
$$
P(f,\Phi)= \sup \{h_{\mu}(f)+\mathcal{F}_*(\mu,\Phi):\mu\in
\mathcal{M}_f,\  \mathcal{F}_*(\mu,\Phi)\neq -\infty\}.
$$
\end{theorem}

We refer the reader to \cite{CFH08,Ba06,Mu06,FH10} for the proof of this variational principle and details
on topological pressure of non-additive potentials. An $f$-invariant probability measure $\mu$ that
attains the supremum is called \emph{equilibrium state} for $f$ with respect to $\Phi$. In many situations
these arise as invariant measures absolutely continuous with respect to weak Gibbs measures.

\begin{definition}\label{def:weak.Gibbs}
Given a sequence of functions $\Phi=\{\varphi_n\}$, we say that a
probability measure $\nu$ is a \emph{weak Gibbs measure with
respect to $\Phi$ on $\Lambda\subset M$}, if the set $\Lambda$ has
full $\nu$-measure  and there exists $\varepsilon_0>0$  such that
for every $x \in \La$ and $0<\vep<\varepsilon_0$ there exists a
sequence of positive constants $(K_n)_{n\geq1}$ (depending only on
$\varepsilon$) satisfying $\lim\limits_{n\to\infty} \frac{1}{n}
\log K_n=0$ and for every $n\ge 1$
\begin{equation*}
K_{n}^{-1}
        \leq \frac{\nu(B_n(x,\vep))}{e^{-n P(f,\Phi)+\varphi_n(x)}} \leq
K_n.
\end{equation*}
We say that $\nu$ is a \emph{Gibbs measure with respect to
$\Phi$}, if there exists $K>0$ such that the same property holds
with $K_n=K$ independent of $n$.
\end{definition}
The previous notion of Gibbs measure is a generalization of the usual one
obtained in \cite{Ba06,Mu06} in the uniformly hyperbolic
setting. In the case of additive potentials,  these weak Gibbs
measures appear in dynamics with some non-uniform hyperbolicity as
e.g. \cite{VV10,Yu00}. Let us focus on some results concerning the existence
of equilibrium states in the uniformly hyperbolic setting.
Given a basic set $\Omega$ for an Axiom A diffeomorphism $f$, it
is known that every almost additive potential $\Phi=\{\varphi_n\}$
satisfying
\begin{enumerate}
\item (bounded variation)
    $\exists A,\delta>0: \;\; \sup_{n\in\mathbb N} \ga_n(\Phi,\de) \le A,$
\end{enumerate}
where $\gamma_n(\Phi, \delta)=\sup\{|\varphi_n(y)-\varphi_n(z)|:\
y,z \in B_n(x,\delta)\}$, admits a unique equilibrium state
$\mu_\Phi$ which coincides with the Gibbs measure w.r.t. $\Phi$
(see ~\cite{Ba06,Mu06} for the proof). This condition above was
introduced by Bowen \cite{Bow74} to obtain uniqueness of
equilibrium states for expansive maps with the specification
property. We say that a sequence of continuous functions
$\Psi=\{\psi_n\}$ satisfy \emph{weak Bowen condition}, if there
exist $\delta>0$ and a sequence of positive real numbers
$\{a_n\}_n$ such that
$\limsup\limits_{n\rightarrow\infty}\frac{a_n}{n}=0$ and
\begin{equation}
\ga_n(\Phi,\de)
    \leq a_n, \quad \text{ for all $n\ge 1$}.
\end{equation}

\section{Statement of the results }\label{s.statements}

Here we state our main results of this paper. The first one is a modified Brin-Katok local entropy formula
for dynamical systems when some errors are admissible. We prove that the exponential decreasing rate
of the measure of the mistake dynamical ball is equal to the measure-theoretic entropy.

\begin{mainproposition}\label{thm:BK}
Given an $f$-invariant ergodic measure $\mu$ and a mistake
function $g$, the following limits
\begin{eqnarray*}
\underline{h}_{\mu}(g;f,x)=\lim_{\varepsilon\to 0}\liminf_{n\to
\infty}-\frac{1}{n} \log \mu(B_n(g;x,\varepsilon))\end{eqnarray*}
and
\begin{eqnarray*}
    \overline{h}_{\mu}(g;f,x)=\lim_{\varepsilon\to 0}\limsup_{n\to\infty}-\frac{1}{n}\log \mu(B_n(g;x,\varepsilon))
\end{eqnarray*}
exist for $\mu$-almost every $x$ and coincide with the measure
theoretic entropy $h_{\mu}(f)$.
\end{mainproposition}

Let us mention that, although the statement of the previous proposition was expected,
the proof of the later
formulas does not follow the original strategy of Brin and Katok. Notice that
the mistake dynamical balls $B_n(g;x,\vep)$ take into
account not only the size $n$ as the amount of allowed mistakes
$g(n,\varepsilon)$. In particular, the mistake dynamical balls
$B_n(g;x,\vep)$ may not even satisfy the inclusion $B_{n+1}(g;x,\vep)\subset
B_n(g;x,\vep)$, e.g, if $g(n,\varepsilon)$ is much larger than
$g(n-1,\varepsilon)$. Since this fact is not standard an estimation on the measure of mistake
dynamical balls for uniformly expanding maps is given in the Appendix~A.
Furthermore, the ergodicity assumption in Proposition~\ref{thm:BK}
is not crucial. Given $\mu\in \mathcal{M}_f$, by ergodic
decomposition theorem we know that $\mu$ can be decomposed as a
convex combination of ergodic measures, $\mu=\int \mu_x
\mathrm{d}\mu(x)$. Applying Proposition~\ref{thm:BK} to each
ergodic component $\mu_x$ and using $h_{\mu}(f)=\int h_{\mu_x}(f)
\mathrm{d}\mu(x)$ we obtain:

\begin{corollary}
Given any $\mu\in \mathcal{M}_f$, the limits $\underline{h}_{\mu}(g;f,x)$ and
$\overline{h}_{\mu}(g;f,x)$ do exist for $\mu$-almost every $x$ and the measure
theoretic entropy $h_\mu(f)$ satisfies
\[
h_{\mu}(f)=\int \underline{h}_{\mu}(g;f,x) \mathrm{d}\mu (x)=\int
\overline{h}_{\mu}(g;f,x) \mathrm{d}\mu(x).
\]
\end{corollary}

\subsubsection*{Large deviation bounds for asymptotically additive observables}

We are also interested to study the rate of convergence at
Kingman's sub-additive theorem. More precisely, given a Borel
probability measure $m$ on the space $M$, we study the rate at
which the $m$-measure of the sets
\begin{equation*}
B(n)=\Big\{
    x\in M \colon  \Big| \frac1n\varphi_n(x) - \cF_*(\mu,\Phi) \Big| > c
    \Big\}
\end{equation*}
goes to zero as $n$ tends to infinite, with respect to our
reference and not necessarily invariant probability measure $m$.
Given a mistake function $g$ and a Borel probability $\nu$, we
define
\[
h_m(g;f, x)=\lim_{\varepsilon\to 0}\limsup_{n\to \infty}
-\frac{1}{n}\log m(B_n(g;,x,\varepsilon))
\]
and $h_m(g; f, \nu)=\nu-\text{ess sup} \, h_m(g;f,x)$. It follows
from Proposition~\ref{thm:BK} that we have
$h_{\nu}(g;f,x)=h_{\nu}(f)$ for $\nu-$a.e. $x$ and every $\nu\in
\mathcal{E}_f$. To provide more precise bounds for $h_m(g;f, x)$,
we introduce two sets of functions as follows.

 Given a constant $K$ and a mistake function $g$, define $\mathcal{V}_K^+(g)$ as the set of
sequences $\Phi\in \mathcal A$ for which there exists $\vep_0>0$
and a set $\Upsilon$ of full $m-$measure such that the following
property holds: for all $0<\vep<\vep_0$, there are constants $C_n$
(depending only on $\varepsilon$) so that
$\lim\limits_{n\to\infty} \frac{1}{n} \log C_n=0$ and
\begin{eqnarray*}
m(B_n(g;x,\varepsilon))\leq C_n \exp(-nK+\varphi_n(x)),
    \quad \forall x\in \Upsilon~\text{and}~n\ge 1.
\end{eqnarray*}
We also consider the set $\mathcal{V}_K^-(g)$ as the set of
sequences $\Phi\in \mathcal A$ for which there exists $\vep_0>0$
and a set $\Upsilon$ of full $m-$measure such that the following
property holds: for all $0<\vep<\vep_0$, there are constants $C_n$
so that $\lim\limits_{n\to\infty} \frac{1}{n} \log C_n=0$ and
$$
m(B_n(g;x,\varepsilon))\geq C_n\exp(-nK+\varphi_n(x)),
    \quad \forall x\in \Upsilon~\text{and}~n\ge 1.
$$ Observe that both classes of asymptotically potentials $\mathcal{V}_K^+(g)$, $\mathcal{V}_K^-(g)$ depend
on the mistake function $g$ and  the constant $K$ used in the
expressions above. For the mistake function $g\equiv 0$, we will
simply write $\mathcal{V}_K^+$ and $\mathcal{V}_K^-$ respectively.
We shall refer to the constants $C_n$ above as tempered constants.
 Finally, for any not necessarily additive sequence
of observables $\Phi=\{\varphi_n\}$ and $E\subset \mathbb{R}$
define
\[
\overline{R}_m(\Phi,E)=\limsup_{n\to \infty} \frac 1 n \log m\left(
\left\{x\in M: \frac 1 n \varphi_n(x)\in E \right\}\right)
\]
and
\[
\underline{R}_m(\Phi,E)=\liminf_{n\to \infty} \frac 1 n \log m\left(
\left\{x\in M: \frac 1 n \varphi_n(x)\in E \right\}\right).
\]
The following abstract results generalize \cite[Theorem A]{You90}
to the case of asymptotically additive potentials under some mistake dynamical systems.

\begin{maintheorem}\label{deviations}
Assume $h_{top}(f)<\infty$. Then for each $\Phi\in \mathcal A$,
$c\in \mathbb{R}$ and mistake function $g$ the following holds:
\begin{enumerate}
\item $\underline{R}_m(\Phi, (c,\infty))\geq \sup
\{h_{\nu}(f)-h_m(g;f,\nu):\nu\in \mathcal{E}_f,~\cF_*(\nu,\Phi)
>c\}$; \item For each $\Psi\in \mathcal{V}_K^+(g)$, we have
\[
\overline{R}_m(\Phi, [c,\infty))\leq  \sup
\{-K+h_{\nu}(f)+\cF_*(\nu,\Psi)\}
\]
where the supremum is taken over all $\nu\in \mathcal{M}_f$
satisfying $\cF_*(\nu,\Phi) \geq c$; \item For every $\Psi\in
\mathcal{V}_K^-(g)$ as above, we have
\[
\underline{R}_m(\Phi, (c,\infty))\geq  \sup
\{-K+h_{\nu}(f)+\cF_*(\nu,\Psi)\}
\]
where the supremum is taken over all $\nu\in \mathcal{E}_f$
satisfying $\cF_*(\nu,\Phi) > c$ and $\nu(\Upsilon)=1$;
 \item Assume $f$
satisfies the $g$-almost specification property. Then, given
$\Psi\in \mathcal{V}_K^-$
\[
\underline{R}_m(\Phi, (c,\infty))\geq \sup
\{-K+h_{\nu}(f)+\cF_*(\nu,\Psi)\}
\]
where the supremum is taken over all $\nu\in \mathcal{M}_f$
satisfying $\cF_*(\nu,\Phi) > c$ and $\nu(\Upsilon)=1$.
\end{enumerate}
\end{maintheorem}

The previous result has particularly interesting applications to
weak Gibbs measures obtained in thermodynamical formalism as we
now describe.

\begin{maintheorem}\label{thm.deviations}
Assume that $\htop(f)<\infty$. Let $\Phi=\{\varphi_n\}$ be an
almost additive family of potentials with $P(f,\Phi)>-\infty$ and let $\nu$
be a weak Gibbs measure for $f$ with respect to $\Phi$ on
$\Lambda\subset M$. Assume that either:
\begin{itemize}
\item[(a)] $\Psi=\{\psi_n\}$ is an asymptotically additive family of potentials, or;
\item[(b)] $\Psi=\{\psi_n\}$ is a sub-additive family of potentials such that:
    \begin{itemize}
    \item[i.] $\Psi=\{\psi_n\}$ satisfies the weak Bowen condition;
    \item[ii.] $\inf_{n\ge 1} \frac{\psi_n(x)}{n}>-\infty$ for all $x\in M$;  and
    \item[iii.] the sequence $\{\psi_n/n\}$ is equicontinuous.
    \end{itemize}
\end{itemize}
Given $c \in \R$, it holds
\begin{align*}
\overline{R}_{\nu}(\Psi, [c,\infty))
    &  \leq \sup \big\{-P(f,\Phi)+h_\eta(f) + \cF_*(\eta,\Phi) \big\} \} \tag{UB}
\end{align*}
where the supremum is over all $\eta\in\cM_f$ such that $\cF(\eta,\Psi) \geq c$.
Moreover,
\begin{align*}
 \underline{R}_{\nu}(\Psi, (c,\infty))
     \geq \sup \big\{-P(f,\Phi)+h_\eta(f) + \cF_*(\eta,\Phi) \big\}
\end{align*}
where the supremum is taken over all \emph{ergodic} measures
$\eta$ satisfying $\cF_*(\eta,\Psi) > c$ and $\eta(\Lambda)=1$.
If, in addition, $f$ satisfies specification property  then
\begin{align*}
\underline{R}_{\nu}(\Psi, (c,\infty))
     \geq \sup \big\{-P(f,\Phi)+h_\eta(f) + \cF_*(\eta,\Phi) \big\} \tag{LB}
\end{align*}
where the supremum is taken over all $\eta\in\cM_f$ satisfying $\cF_*(\eta,\Psi) > c$ and $\eta(\La)=1$.
\end{maintheorem}

Some comments on our assumptions are in order. We use distinct strategies to deal with the two
different classes of potentials $\Psi$.  On the one hand, given a asymptotically additive family of potentials
$\Psi=\{\psi_n\}$ the potentials $\psi_n$ can be approximated by Birkhoff sums of continuous potentials and,
in particular, the functional $\mu \to \cF_*(\mu,\Psi)$ is continuous.
On the other hand, for the sub-additive setting conditions (ii) and (iii) will imply the continuity of the previous
functional and condition (i) is a bounded distortion property as explained before. In Example~\ref{Ex:Holder}
we explain how to deduce conditions (i) and (iii) in the expanding setting if the family $\Psi$ satisfies
some H\"older continuous regularity.
Let us also mention that for the lower bound estimate above it is enough the measure $\mu$ to satisfy
the non-uniform specification property defined in \cite{STV03,Va12}, but we shall not use or prove this fact here.
Under the uniform hyperbolic assumption, we can also extend a
large deviations principle to this non-additive setting as
follows.

\begin{maincorollary}\label{cor:LDP}
Assume that $\Omega$ is a basic set for a uniform hyperbolic  map
$f$, that $\Phi=\{\varphi_n\}$ is an almost additive sequence of
functions such that there exists $\nu_\Phi$ a Gibbs measure for
$\Phi$ and $\mu_\Phi\ll \nu_\Phi$ a unique equilibrium state for
$f$ with respect to $\Phi$. If $\Psi=\{\psi_n\}$ is a family of
potentials as in Theorem~\ref{thm.deviations}, then it satisfies
the following large deviations principle: given $c \in \R$ it
holds that
\begin{align*}
\overline{R}_{\nu_\Phi}(\Psi, [c,\infty))   \leq - \inf_{\eta\in\cM_f}
                \left\{P(f,\Phi)-h_\eta(f) - \cF_*(\eta,\Phi) \colon |\cF_*(\eta,\Psi)-\cF_*(\mu_\Phi,\Psi)|\ge c \right\}
\end{align*}
and also
\begin{align*}
\underline{R}_{\nu_\Phi}(\Psi, (c,\infty))  \geq - \inf_{\eta\in\cM_f}
                \left\{P(f,\Phi)-h_\eta(f) - \cF_*(\eta,\Phi) \colon |\cF_*(\eta,\Psi)-\cF_*(\mu_\Phi,\Psi)|> c \right\}.
\end{align*}
\end{maincorollary}

Let us finish this section with some comments. First notice that
the uniqueness of the equilibrium state implies that  the later
convergence to the average in Kingman's subadditive ergodic
theorem is indeed exponential.
Second, since almost additive potentials are  indeed
asymptotically additive then our result apply for a wide class of
non-additive sequences of observables.

\section{Examples and applications}\label{s:examples}

In this section we discuss the specification properties and large deviation results on
a broad class of examples. In particular, we obtain applications to the study of the velocity
of convergence for the Lyapunov exponents in open classes of linear cocycles and non-conformal
expanding repellers. Our first example concerns sub-additive families of H\"older continuous potentials.

\begin{example}\label{Ex:Holder}
Let $X$ be a compact metric space and assume that $f:X\to X$ expands distances, that is,
there are $\la > 1$ and $\vep>0$ such that $d(f(x),f(y))\ge \la \, d(x,y) $ for all $y\in B(x,\vep)$ and, consequently,  $f^n:B_n(x,\vep)\to B(f^n(x),\vep)$ is a bijection.
Let $\Psi=\{\psi_n\}_n$ be any sub-additive family of $\ga$-H\"older continuous potentials
such that the H\"older constants have at most linear growth, meaning that there exists
$K>0$ such that $\text{H\"ol}_\ga(\psi_n)\le K n$.
We claim that $\Psi$ satisfies assumptions (i) and (iii) in
Theorem~\ref{thm.deviations}. On the one hand, $
|\psi_n(x)-\psi_n(y)| \leq  K n \, d(x,y)^\ga \leq K n \la^{-\ga
n} \diam(X)^\ga $ for every $y\in B_n(x,\vep)$ and $n\ge 1$. In
consequence, $\ga_n(\Phi,\vep)\le a_n= K n \la^{-\ga n}
\diam(X)^\ga$ where  $\lim\limits_{n\rightarrow\infty} a_n/n=0$.
This proves that $\Psi$ satisfies the weak Bowen condition.
On the other hand, the sequence $\{\psi_n/n\}$ is H\"older
continuous with uniform constant $K$. Thus this sequence is
equicontinuous.
\end{example}

Now we provide an illustrative example of a transformation that
does satisfy the  almost specification property but does not
satisfy the strong specification property.

\begin{example}\label{ex.beta.asymp}
Consider the piecewise expanding maps of the interval $[0,1)$
given by $T_\beta(x)=\beta x (\!\!\!\mod 1)$, where $\beta>1$.
This family is known as \emph{beta transformations} and it was
introduced by R\'enyi in \cite{re}. It was proved by
Buzzi~\cite{Buz97} that for all but countable many values of
$\beta$ the transformation $T_\beta$ does not satisfy the
specification property.  However, it follows from \cite{ps,Th10} that every
$\beta$-map satisfies the almost specification property for every
unbounded mistake function $g$. It is well known that for all $\beta>1$ the topological entropy
of $T_\beta$ is $\log\beta$ and it admits a unique  maximal entropy measure $\mu_\beta$.
Moreover, Pfister and Sullivan~\cite{ps} proved a level-2 large deviations principle for
the $\beta$-transformations with respect to $\mu_\beta$. In fact, a key point is that the maximal
entropy measure $\mu_\beta$ satisfies a weak Gibbs property (see e.g. \cite[Equation~(5.18)]{ps0}), that is,
there are positive constants $K_n$ so that $\lim_{n\to\infty}\frac1n\log K_n=0$ and  for all $n\ge 1$
$$
K_n^{-1} \beta^{-n}
    \leq \mu_\beta(B_n(x,\vep))
    \leq K_n \beta^{-n}.
$$
Since the discontinuities have zero entropy, our results apply
also in this piecewise expanding setting. Thus, if $\Phi=\{ 0\}$
and $\Psi=\{\psi_n\}$ is any family of asymptotically additive
potentials (e.g. $\psi_n=\sum_{j=0}^{n-1}\psi\circ T_\beta^j +a_n$
where $\psi$ is a continuous function and $(a_n)_n$ is any
sequence of real numbers so that $a_n/n$ is convergent to zero)
then Theorem~\ref{thm.deviations} and
Proposition~\ref{prop:measureballs} yields that for any $c\in
\mathbb{R}$
\[
\limsup_{n\to \infty} \frac 1 n \log \mu_\beta \left( \left\{x\in M: \frac 1 n \psi_n(x)\ge c \right\}\right)
    \leq \sup_\eta \big\{-\log \beta+h_\eta(T_\beta) \big\}
\]
where the supremum is taken over all $\eta\in
\mathcal{M}_{T_\beta}$ satisfying $\cF_*(\eta,\Psi) \geq c$, and
also
\[
\limsup_{n\to \infty} \frac 1 n \log \mu_\beta \left( \left\{x\in M: \frac 1 n \psi_n(x)\ge c \right\}\right)
    \geq  \sup_\eta \{-\log\beta+h_{\eta}(T_\beta)\}
\]
where the supremum is taken over all $\eta\in
\mathcal{M}_{T_\beta}$ satisfying $\cF_*(\eta,\Psi) > c$. Finally,
one should mention that Climenhaga, Thompson and
Yamamoto~\cite{CTY} have recently obtained a level-2 large
deviations principle for symbolic systems equiped with a large
class of reference measures, that include the
$\beta$-transformations.
\end{example}

In the next example we provide an application of our results in the uniformly expanding context.

\begin{example}
Let $f:M\rightarrow M$ be a $C^1$ map, and let $J\subset
M$ be a compact $f$-invariant set. If $J$ is a maximal topological
mixing repeller, Barreira~\cite[Page 289]{Ba06}  proved that each
almost additive potential $\Phi=\{\varphi_n\}$ with weak Bowen
property has a weak Gibbs measure $\nu_\Phi$. Thus Theorem B
applies to $\nu_\Phi$ and any asymptotically additive family of potentials $\Psi$.

We also point out every almost additive potential indeed satisfies
the weak Bowen condition, see \cite[Lemma 2.1]{ZZC11} for a proof.
Thus, any almost additive potential has a weak Gibbs measure.
Therefore, if $J$ is a maximal topological mixing repeller,
Theorem B applies to any almost additive family of potentials $\Phi$ and any
asymptotically additive family of potentials $\Psi$ without any additional conditions.
\end{example}

In our next class of examples we estimate the rate of convergence of the maximal Lyapunov exponent
for an important open class of linear cocycles.

\begin{example}
Here we consider cocycles over subshifts of finite type considered by Feng and Lau~\cite{FL02} and
later by Feng and K\"aenm\"aki~\cite{FK12}.  Let $\si:\Si \to \Si$ be the shift map on the space
$\Si=\{1,\dots, \ell\}^{\mathbb N}$ endowed with the distance $d(x,y)=2^{-n}$ where $x=(x_j)_j$,
$y=(y_j)_j$ and $n=\min \{j \ge 0 : x_j\neq y_j\}$.
Consider matrices $M_1, \dots, M_\ell \in \cM_{d\times d}(\mathbb C)$ such that
for every $n\ge 1$ there exists $i_1,\dots, i_n\in\{1,\dots, \ell\}$ so that the product matrix
$M_{i_1}\dots M_{i_n}\neq 0$. Then, the topological pressure function is well defined as
\begin{equation*}\label{eq.pressureLau}
P(q)=\lim_{n\to\infty} \frac1n \log \sum_{\iota \in \Si_n} \|M_\iota\|^q
\end{equation*}
where $\Si_n=\{1,\dots, d\}^n$ and for any $\iota=(i_1, \dots, i_n) \in \Si_n$ one considers the matrix
$M_\iota = M_{i_n} \dots M_{i_2} M_{i_1}$. Moreover, for any $\si$-invariant probability measure $\mu$
define also the maximal Lyapunov exponent of $\mu$ by
\begin{equation*}\label{eq.LyapunovLau}
M_*(\mu)=\lim_{n\to\infty} \frac1n \sum_{\iota\in\Si_n} \mu([\iota]) \log \|M_\iota\|
\end{equation*}
and it holds that $P(q) = \sup\{ h_\mu(\si) + q \, M_*(\mu) : \mu
\in \cM_\si \}$. Notice that this is the variational principle for
the  potentials $\Psi=\{\psi_n\}$ where $\psi_n(x)= q \log \|
M_{\iota_n(x)}\|$ and for any $x\in \Si$ we set
${\iota_n(x)}\in\Si_n$ as the only symbol such that $x$ belongs to
the cylinder $[\iota_n(x)]$.
Assume that the set of matrices $\{M_1, \dots, M_d\}$ is irreducible over $\mathbb C^d$, that is, there is no
non-trivial subspace $V\subset \mathbb C^d$ such that $M_i(V)\subset V$ for all $i=1,\dots, \ell$. Then it follows
from \cite[Proposition~1.2]{FK12} that there exists a unique equilibrium state $\mu_q$ for $\si$ with respect to
$\Psi$ and it is a Gibbs measure: there exists $C>0$ such that
\begin{equation*}\label{eq:GibbsLau}
\frac1C
    \leq \frac{\mu_q([\iota_n])}{e^{-n P(q)} \|M_{\iota_n} \|^q   }
    \leq C
\end{equation*}
for all $\iota_n\in \Si_n$ and $n\ge 1$.

Moreover, it is not hard to check that for any $\vep>0$ we get $y\in B_n(x,\vep)$ if and only if the sequences
$x=(x_j)$ and $y=(y_j)_j$ verify $x_j=y_j$ for all $0\le j \le n+[ \frac{-\log\vep}{\log 2} ] $. In consequence,
$B_n(x,\vep)\subset [\iota_n(x)]$ where the potential $\log \| M_{\iota_n(x)}\|$ is constant.
Therefore, the sub-additive family of potentials $\{\log \| M_{\iota_n(x)}\|\}$ clearly satisfies the weak Bowen condition,
and Corollary~\ref{cor:LDP} yields that for any $\de>0$
$$
\mu_q\Big (x\in \Si : \Big|\frac1n \log \|M_{\iota_n(x)}\| -M_*(\mu)\Big|>\de \Big)
$$
decreases exponentially fast.
\end{example}

The next example combines the theory for both additive and non-additive
families of potentials in a non-uniformly expanding context.

\begin{example}
Let $f$ be the Manneville-Pomeau map on the interval $[0,1]$ given by
$f(x)=x+x^{1+\al} (\!\!\! \mod 1)$, for $\al\in(0,1)$. This transformation satisfies the
specification property since it is topologically conjugated to the doubling map.
Moreover, it is well known that there exists an  equilibrium state $\mu \ll \Leb$ for $f$
with respect to the potential $\phi=-\log |f'|$ and there exists a sequence $K_n$ so that
$\limsup_{n\to\infty}\frac1n\log K_n=0$ and such that the measure has the weak Gibbs property:
$$
\frac1{K_n}
    \leq \frac{\mu(\cP^{(n)}(x))}{|(f^n)'(x)|}
    \leq K_n
$$
for all $x\in [0,1]$ and $n\ge 1$, where $\cP$ is the Markov
partition for $f$, $\cP^{(n)}:=\bigvee_{j=0}^{n-1} f^{-j}\cP$ and
$\cP^{(n)}(x)$ is the element of the partition $\cP^{(n)}$ that
contains $x$. So, our results apply for any family of
asymptotically additive potentials or sub-additive potentials with
the weak Gibbs property $\Psi=\{\psi_n\}$.

Let us mention that Yuri \cite{Yu00} proved that typical piecewise
$C^1$-smooth maps $f$ with indifferent periodic points admit
invariant ergodic weak Gibbs measures for $-\log |\det Df|$. In
particular our large deviation  upper bound results also hold in
this context.
\end{example}

The following class of local diffeomorphisms was introduced by Barreira and Gelfert~\cite{BG06}
in the study of multifractal analysis for Lyapunov exponents associated to non-conformal repellers.

\begin{example} \label{bg:almostadditive}
Let $f:\mathbb{R}^2\rightarrow\mathbb{R}^2$ be
a $C^1$ local diffeomorphism, and let $J\subset \mathbb{R}^2$ be a
compact $f$-invariant set. Following \cite{BG06}, we say that $f$ satisfies the following {\em cone condition}
on $J$ if there exist a number $b\leq 1$ and for each $x\in J$ there is a one-dimensional subspace
$E(x)\subset T_x\mathbb{R}^2$ varying continuous with $x$ such that
$$Df(x)C_b(x)\subset \{0\}\cup\mathrm{int} \, C_b(fx)$$
where $C_b(x)=\{(u,v)\in E(x)\bigoplus E(x)^{\bot}:~||v||\leq
b||u||\}$.
It follows from \cite[Proposition~4]{BG06} that the later condition implies that both  families of potentials
given by $\Phi_1=\{\log \sigma_1(Df^n(x))\}$ and $\Phi_2=\{\log \sigma_2(Df^n(x)) \}$ are almost additive,
where $\sigma_1(L)\geq \sigma_2(L)$ stands for the singular values of the linear transformation
$L:\mathbb{R}^2\rightarrow\mathbb{R}^2$, i.e., the eigenvalues of $(L^{*}L)^{1/2}$  with $L^*$
denoting the transpose of $L$.
Assume that $J$ is a locally maximal topological mixing repeller of $f$ such that:
\begin{itemize}
\item[(i)] $f$ satisfies the cone condition on $J$, and
\item[(ii)] $f$ has bounded distortion on $J$, i.e., there exists some $\delta>0$ such that
$$ \sup_{n\ge 1}\frac 1n \log \sup\Big\{ ||Df^n(y)(Df^n(z))^{-1}||:~x\in J~\text{and}~y,z\in B_n(x,\delta) \Big\}<\infty.$$
\end{itemize}
Then it follows from \cite[Theorem 2]{BG06} that there exists a
weak Gibbs measure $\nu_{\sigma_i}$ with respect to the family
of potentials $\Phi_i$, for $i=1,2$. Moreover, \cite[Theorem~9]{Ba06} yields that there exists a
unique equilibrium state $\mu_i$ for $(f,\Phi_i)$.
So, for any $i,j\in\{1,2\}$ it follows from Theorem~\ref{thm.deviations} applied to the weak Gibbs measure
$\nu_{\sigma_i}$ and family of potentials $\Psi=\Phi_j=\{\log \sigma_j(Df^n(x))\}$  that
\begin{align*}
\limsup_{n\to \infty} \frac 1 n \log \nu_{\sigma_i} &
    \left(
    \left\{x\in M: \frac 1 n \log \sigma_j(Df^n(x)) \ge c \right\}\right)\\
    & \leq \sup \left\{-P(f,\Phi_i)+h_\eta(f) + \lim_{n\to\infty} \frac1n \int \log \sigma_i(Df^n(x)) d\eta \right\}
\end{align*}
where the supremum is over all $\eta\in\cM_f$ such that $\inf_{n\ge 1} \frac1n \int \log \sigma_j(Df^n(x)) d\eta\ge c$.
and also
\begin{align*}
\liminf_{n\to \infty} \frac 1 n \log \nu_{\sigma_i} &
    \left(
    \left\{x\in M: \frac 1 n \log \sigma_j(Df^n(x)) > c \right\}\right)\\
    & \geq \sup \left\{-P(f,\Phi_i)+h_\eta(f) + \lim_{n\to\infty} \frac1n \int \log \sigma_i(Df^n(x)) d\eta \right\}
\end{align*}
where the supremum is taken over all $\eta\in\cM_f$ so that
$\inf_{n\ge 1} \frac1n \int \log \sigma_j(Df^n(x)) d\eta > c$. As
a simple application of the previous formulas, we get that for any
$\vep>0$ the tail of the convergence to the largest or smallest
Lyapunov exponent (corresponding respectively to $j=1$ or $j=2$)
\begin{align*}
\nu_{\sigma_i}
    \left(
    \left\{x\in M: \left|\frac 1 n \log \sigma_j(Df^n(x)) - \lim_{n\to\infty} \frac 1 n \int \log \sigma_j(Df^n(x)) d\mu_i
    \right|>\vep \right\}\right)
\end{align*}
decays exponentially fast as $n\to\infty$.
\end{example}

Finally, in our last example we deal with an example where both families of potentials
responsable by computing the largest and smaller Lyapunov exponents are asymptotically additive.

\begin{example}\label{acr}
Let $M$ be a $d$-dimensional smooth  manifold and $J$ a compact
expanding invariant set for a $C^1$ map $f$. We say that $J$ is an
average conformal repeller if all Lyapunov exponents of each
ergodic measure are equal and positive. In particular, it follows
from \cite[Theorem~4.2]{BCH10} that
$$
\lim_{n\to\infty} \frac1n  \left( \log \|Df^n(x)\| - \log \|Df^n(x)^{-1}\|^{-1} \right)
    = \lim_{n\to\infty} \frac1n  \log \frac{\|Df^n(x)\|}{\|Df^n(x)^{-1}\|^{-1}}
    = 0
$$
uniformly on $J$. It is easy to see that  the family of continuous
potentials $\Psi_1=\{\log \|Df^n(x)\|\}_n$ is sub-additive while
$\Psi_2=\{\log \|Df^n(x)^{-1}\|^{-1}\}_n$ is superadditive.
Furthermore, in this setting it is not hard to check that these
two families of potentials are asymptotically additive since they
can be uniformly approximated by the additive potentials $\{\frac
1d\log |\det(Df^n(x))|\}_n$.
Let $\Phi$ be any family of continuous potentials such that $f$
has a unique equilibrium state $\mu_\Phi$ for $f$ with respect to
$\Phi$, and  $\mu_\Phi$ satisfies the Gibbs property. Then it
follows from Corollary~\ref{cor:LDP} that
$$
\mu_\Phi \Big(
    x\in J : \Big| \frac1n \log \|Df^n(x)\| - \la(\mu_\Phi) \Big|>\de
    \text{ or }
    \Big|  \frac1n \log \|Df^n(x)^{-1}\|^{-1} - \la(\mu_\Phi) \Big|>\de
    \Big)
$$
decrease exponentially fast, where in this average conformal setting we consider
$\la(\mu_\Phi)=\inf_n \frac1n \int \log \|Df^n(x)\| \, d\mu_\Phi= \sup_n \frac1n \int -\log \|Df^n(x)^{-1}\| \, d\mu_\Phi$
as the average Lyapunov exponent for $\mu_\Phi$.
A final remark is that $f$ admits equilibrium states with respect
to the  potentials $\Psi_1$ and $\Psi_2$ by upper semi-continuity
of the metric entropy.
\end{example}

\section{Proof of the main results}\label{s.proof1}

\subsection{Proof of Proposition~\ref{thm:BK}}

We first recall a covering lemma for mistake dynamical
balls of points with slow recurrence to the boundary of a partition.

\begin{lemma}\cite[Lemma~3.2]{rvz} \label{yl32} Let $\mathcal{Q}$ be a finite partition
of $M$ and consider $\varepsilon>0$ arbitrary small. Let
$V_\varepsilon$ denote the $\varepsilon$-neighborhood of the
boundary $\partial\mathcal{Q}$. For any $\alpha>0$, there exists
$\gamma>0$(depending only on $\alpha$),  such that for every $x\in
M$ satisfying $\sum_{j=0}^{n-1}\chi_{V_\varepsilon}(f^jx)<\gamma
n$, the mistake dynamical ball $B_n(g;x,\varepsilon)$ can be
covered by $e^{\alpha n}$ cylinders of $\mathcal{Q}^{(n)}$ for
sufficiently large $n$.
\end{lemma}

We are now in a position to  prove our generalized Brin-Katok
local entropy formula whose proof exploits the ergodicity of the
measure.

\begin{proof}[Proof of Proposition~\ref{thm:BK}]

First we note that the limits in the statement of
Proposition~\ref{thm:BK} are indeed well defined almost
everywhere. Given $n\geq 1, \varepsilon>0$ and $x\in M$ it is
clear  that  $B_n(x,\vep)\subset B_n(g;x,\varepsilon)$. Thus
Brin-Katok formula (see \cite{bk}) immediately yields
\begin{equation*}
\lim_{\vep\to 0} \limsup_{n\to\infty} -\frac1n \log \mu(B_n(g;x,\varepsilon))
    \leq \lim_{\vep\to 0} \limsup_{n\to\infty} -\frac1n \log \mu(B_n(x,\varepsilon))
    =h_\mu(f)
\end{equation*}
for almost every $x$. Hence, to complete the proof of the proposition it is enough to show that
for $\mu$-almost every $x$ one has
\begin{equation*}
\underline{h}_{\mu}(g; f, x)
    = \lim_{\vep\to 0} \liminf_{n\to\infty} -\frac1n \log \mu(B_n(g;x,\varepsilon))
    \geq h_\mu(f).
\end{equation*}
Fix $\alpha>0$ arbitrary and let $\gamma$ be given by
Lemma~\ref{yl32}. Consider a finite partition $\mathcal Q$ of $M$
such that $\mu(\partial \cQ)=0$ and $h_\mu(f)\leq
h_\mu(f,\cQ)+\gamma$. If $\vep>0$ is small enough the
$\vep$-neighborhood $V_\vep$ of $\partial Q$ satisfies
$\mu(V_\vep)<\gamma/2$. For each positive integer $N$ set
\[
\Gamma_N=\Big\{x\in M: \frac1n
\sum_{j=0}^{n-1}\chi_{V_\varepsilon}(f^j(x))
<\gamma~\text{and}~\mu(\cQ^{(n)}(x)) \leq e^{-n
(h_\mu(f,\cQ)-\gamma)},~\forall n\geq N \Big\}.
\]
Using the ergodicity of $\mu$ it follows that $\Gamma_N\subset
\Gamma_{N+1}$ and $\mu(\cup_N \Gamma_N)=1$.
By Lemma~\ref{yl32} one has that for any $x\in \Gamma_{N}$ the
mistake dynamical ball $B_n(g;x,\varepsilon)$ can be covered by
$e^{\alpha n}$ cylinders of $\mathcal{Q}^{(n)}$.
Therefore we obtain that
\begin{align*}
\mu(B_n(g; x,\varepsilon)\cap \Gamma_N)
        & \leq \mu(\cup \{Q\in \cQ^{(n)}: B_n(g;x,\vep)\cap Q \neq \emptyset
            \text{ and } Q \cap \Gamma_N\neq \emptyset\}) \\
        & \leq e^{\al n} e^{-n h_\mu(f,\mathcal{Q})+\gamma n }
        \leq e^{\al n} e^{-n h_\mu(f)+2\gamma n }
\end{align*}
for all $x\in \Gamma_N$ and $n\ge N$. For each $\varepsilon>0$ and positive
integer $n$ fixed it follows that $\mu(B_n(g;x,\vep)\cap \Gamma_{N}) \to \mu(B_n(g;x,\vep))$ as
$N$ tends to infinite. Therefore, for any arbitrary constant $\xi>0$ it follows that
$
\mu(B_n(g;x,\vep)) \leq e^{\xi} \mu(B_n(g;x,\vep)\cap \Gamma_{N})
$
provided that $N$ is large  and $x\in \Gamma_{N}$. In particular, if $N$ is fixed as above then
$$
\liminf_{n\to\infty} -\frac1n \log \mu(B_n(g;x,\vep)) \geq
h_\mu(f)-2\gamma-\al -\xi.
$$
Considering $\vep$ small enough the constants involved above
converge to zero proving that
$\lim\limits_{\vep\to0}\liminf\limits_{n\to\infty} -\frac1n \log
\mu(B_n(g;x,\vep))\geq h_\mu(f)$ as claimed.
\end{proof}

\subsection{Proof of Theorem~\ref{deviations}}

This subsection is devoted to the proof of Theorem~\ref{deviations}.

\begin{proof}[Proof of Theorem \ref{deviations} (1)] Pick $\nu\in
\mathcal{E}_f$ with $\cF_*(\nu, \Phi)>c$ and let $U_n=\{x\in M:
\frac 1 n \varphi_n(x)>c\}$. It suffices  to show that
\[
\liminf_{n\to \infty}\frac 1 n \log m (U_n) \geq
h_{\nu}(f)-h_m(g;f,\nu).
\]
Without loss of generality, assume $h_m(g;f,\nu)<\infty$,
otherwise, there is nothing to prove since $h_{\nu}(f)\leq
h_{top}(f)<\infty$. The strategy for the proof is to approximate
the family of asymptotically additive functions by appropriate
Birkhoff sums associated to some continuous observable. Indeed,
fix a sufficiently small $\xi>0$ so that $\cF_*(\nu,
\Phi)>c+2\xi$, and $\varphi_\xi$ is given by \eqref{eq.asymptotic}
approximating $\Phi$. We have
\[
\widetilde{U}_n
    =\left\{x\in M: \frac 1 n S_n\varphi_\xi(x)>c+\xi\right\}
    \subset U_n
\]
for all sufficiently large $n$. Moreover, if $\de>0$ is small enough we can
assume that $\cF_*(\nu,\Phi)>c+2\xi+\delta$.
By uniform continuity of $\varphi_\xi$, there exist
$\varepsilon_\delta>0$ such that
$|\varphi_{\xi}(x)-\varphi_{\xi}(y)|<\delta/2$ whenever
$d(x,y)\leq\varepsilon_\de$. We claim that this guarantees that if
$\frac 1 n S_n\varphi_\xi(x)>c+\xi+\delta$ then
$B_n(g;x,\varepsilon)\subset \widetilde{U}_n\subset U_n$ for all
sufficiently large $n$. Indeed, for each $y\in
B_n(g;x,\varepsilon)$ there exists $\Lambda\subset
I(g;n,\varepsilon_0)$ so that $y\in B_\Lambda(x,\varepsilon)$, and
so
\begin{eqnarray} \label{part1}
\sum_{i=0}^{n-1} \varphi_\xi(f^ix)
    \leq \sum_{i\in\Lambda} [ \varphi_\xi(f^iy)+\frac{\delta}{2}]
    +\sum_{i\notin \Lambda}||\varphi_\xi||
    \leq  \sum_{i=0}^{n-1} [ \varphi_\xi(f^iy)+\frac{\delta}{2}]
    +Cg(n,\varepsilon)
\end{eqnarray}
where $C=2||\varphi_\xi||+\delta$. This implies that
$S_n\varphi_\xi(y)\geq S_n\varphi_\xi(x)- C g(n,\varepsilon) - \delta n/ 2 $.
By the definition of mistake function, we have $\frac 1 n S_n\varphi_\xi(y)> c+\xi$
for all sufficiently large $n$, which means that $y\in \widetilde{U}_n\subset U_n$ as claimed.

Fix now an arbitrarily small $\gamma>0$.
 By the modified Katok entropy formula \eqref{eq.mistKatok} we can choose
$0<\varepsilon_1<\varepsilon_\delta$ such that
$
\liminf_{n\to \infty}\frac 1 n \log N(g;n,2\varepsilon,\frac 1 2)
\geq h_{\nu}(f)-\gamma
$
for all $\varepsilon<\varepsilon_1$.
Pick also $0<\varepsilon_2<\varepsilon_1$ such that
\begin{eqnarray}\label{A1.1}
\nu\left(\Big\{x\in M: \limsup_{n\to \infty}-\frac 1 n \log
m(B_n(g;x,\varepsilon_2)) \leq h_{m}(g;f,\nu)+\gamma\Big\}\right)>\frac 2
3.
\end{eqnarray}
Since $\int \varphi_\xi \mathrm{d}\nu>\cF_*(\nu,
\Phi)-\xi>c+\xi+\delta$, one can use (\ref{A1.1}) and Birkhoff
ergodic theorem to choose a measurable set $\mathcal{L}\subset M$
with $\nu(\mathcal{L})\geq \frac 1 2$ and a positive integer $N$
such that for any $n\geq N$ the following two properties hold for
each $x\in \mathcal{L}$:
\begin{itemize}
\item[(i)] $\frac 1 n S_n\varphi_\xi(x)>c+\xi+\delta$;
\item[(ii)] $m(B_n(g;x,\varepsilon_2))\geq \exp(-n(h_{m}(g;f,\nu)+2\gamma))$.
\end{itemize}
For each sufficiently large $n$, let $E_n$ be a maximal set of
$(g;n,\varepsilon_2)$-separated points contained in
$\mathcal{L}$. It is not hard to check that $\mathcal{L}\subset
\bigcup_{x\in E_n}B_n(g;x,2\varepsilon_2)$ and that the mistake
dynamical balls $\{B_n(g;x,\varepsilon_2): x\in E_n\}$ are
disjoint. Using that $\sharp E_n\geq N(g;n,2\varepsilon_2,\frac 1
2)$ it follows that

\begin{align*}
\liminf_{n\to \infty}\frac 1 n \log m(U_n)
    &\geq  \liminf_{n\to \infty}\frac 1 n \log m(\widetilde{U}_n)\\
    &\geq \liminf_{n\to \infty}\frac 1 n \log \sum_{x\in E_n} m(B_n(g;x,\varepsilon_2))\\
    &\geq \liminf_{n\to \infty}\frac 1 n \log \sharp E_n \exp(-n(h_m(g;f,\nu)+2\gamma))\\
    &\geq h_{\nu}(f)-h_m(g;f,\nu)-3\gamma.
\end{align*}
The arbitrariness of $\gamma$ implies the desired
result.\end{proof}

\begin{proof}[ Proof of Theorem \ref{deviations}(2)] Let
$U_n=\{x\in M: \frac 1 n \varphi_n(x)\geq c\}$ and
$\Psi=\{\psi_n\}\in \mathcal{V}_K^+(g)$. By the definition of
$\mathcal{V}_K^+(g)$, there exists a set $\Upsilon$ of full
$m$-measure and constants $C_n$ and $\varepsilon$
such that
$
m(B_n(g;x,\varepsilon))\leq C_n \exp(-nK+\varphi_n(x))
$
for all $x\in \Upsilon$ and $n\ge 1$.
We will assume without loss of generality that $U_n\subset \Upsilon$,
since otherwise we just consider $U_n\cap \Upsilon$.
It suffices to construct a measure $\nu\in \mathcal{M}_f$ with
$\cF_*(\nu, \Phi)\geq c$ such that
\[
\limsup_{n\to \infty} \frac 1 n \log m(U_n)\leq -K+
h_{\nu}(f)+\cF_*(\nu, \Psi).
\]
Let $E_n$ be a maximal $(g;n,\varepsilon)-$separated
set contained in $U_n$. Note that every
$(g;n,\varepsilon)$-separated set is also a
$(n,\varepsilon)$-separated set.  If $E_n$ is not a maximal
$(n,\varepsilon)$-separated set consider a maximal
$(n,\varepsilon)$-separated set $E'_{n}$ containing $E_n$. Now,
define the probability measures
\[
\mu_n=\sum_{x\in E'_{n} }\frac{e^{\psi_n(x)}}{Z_n}\cdot \frac 1 n
\sum_{i=0}^{n-1}\delta_{f^ix}
\]
where $Z_n=\sum_{x\in E'_{n}} e^{\psi_n(x)}$ and let $\nu$ be a weak$^*$ limit of $\mu_n$.
It is not hard to check that $\nu\in \mathcal{M}_f$. Moreover, it follows from the proof of
\cite[Theorem~1.1]{CFH08} that $ \limsup_{n\to \infty} \frac 1 n\log Z_n\leq h_{\nu}(f)+\cF_*(\nu, \Psi)$.
Now, since $U_n\subset \bigcup_{x\in E_n}B_n(g;x,\varepsilon)$
and the constants $C_n$ are tempered we get
\begin{align*}
\limsup_{n\to \infty}\frac 1 n \log m(U_n)&
    \leq \limsup_{n\to \infty} \frac 1 n \log \sum_{x\in E_n}m(B_n(g;x,2\varepsilon)) \\
    &\leq \limsup_{n\to \infty} \frac 1 n \log \sum_{x\in E_n} C_n \exp(-nK+\psi_n(x))\\
    &\leq -K+ \limsup_{n\to \infty} \frac 1 n \log Z_n \\
    &\leq -K+h_{\nu}(f)+\cF_*(\nu, \Psi).
\end{align*}
It remains to prove that $\cF_*(\nu,\Phi)\ge c$. For each small $\xi>0$, let $\varphi_\xi$
be given by \eqref{eq.asymptotic} approximating $\Phi$. Since $\mu_n$ is a linear combination of
measures $\frac 1 n \sum_{i=0}^{n-1}\delta_{f^ix}$ for $x\in E'_{n}$ and
$$
\int \varphi_\xi \, \mathrm{d}\Big(\frac 1 n \sum_{i=0}^{n-1}\delta_{f^ix}\Big)
    =\frac 1 n S_n\varphi_\xi (x)
    \geq \frac 1 n \varphi_n(x)-\xi
    \geq c-\xi,
$$
we deduce that $\int \varphi_\xi \, d\mu_n \geq c-\xi$. In
consequence $ c-\xi\leq \int \varphi_\xi \mathrm{d}\nu \leq \cF_*(\nu, \Phi)+\xi$.
The arbitrariness of $\xi$ implies that $\cF_*(\nu, \Phi)\geq c$.
This completes the proof of the second part of the
theorem.\end{proof}

\begin{proof}[Proof of Theorem \ref{deviations} (3)]
Let $U_n=\{x\in M: \frac 1 n \varphi_n(x)\geq c\}$ and
$\Psi=\{\psi_n\}\in \mathcal{V}_K^-(g)$.  Without loss of
generality, assume  that $U_n\subset \Upsilon$, since otherwise we
just consider $U_n\cap \Upsilon$. Pick an arbitrary $f$-invariant
and ergodic measure $\nu$ satisfying that $\cF_*(\nu,\Phi)>c$ and
$\nu(\Upsilon)=1$. It follows from the definition of
$\mathcal{V}_K^-(g)$ and the ergodicity of $\nu$ that
\[
\lim_{\vep\rightarrow 0}\limsup_{n\rightarrow\infty}-\frac 1 n
\log m(B_n(g;x,\vep)) \leq K-\cF_*(\nu,\Psi), \quad
\nu-a.e.\  x\in \Upsilon,
\]
and, consequently, $h_m(g;f,\nu)\leq K-\cF_*(\nu,\Psi)$. Therefore, it follows from
the first item of this theorem that
$
\underline{R}_m(\Phi, (c,\infty))  \geq \sup\{-K+h_\nu(f)+\cF_*(\nu,\Psi)\},
$
where the supremum is taken over all $\nu\in \mathcal{E}_f$
satisfying that $\cF_*(\nu,\Phi)>c$ and $\nu(\Upsilon)=1$. This proves the third assertion
of the theorem.
\end{proof}

 \begin{proof}[ Proof of Theorem \ref{deviations}(4)]
 Fix an asymptotically additive family of potentials $\Psi\in \mathcal{V}_K^-$.
Let $U_n=\{x\in M: \frac 1 n \varphi_n(x)>c\}$ and assume, without loss of generality,
that $U_n\subset \Upsilon$.
Pick an invariant measure $\nu\in \mathcal{M}_f$ with $\cF_*(\nu,
\Phi)>c$ and $\nu(\Upsilon)=1$. It suffices to prove that
\[
\liminf_{n\to \infty}\frac 1 n \log m(U_n)\geq -K+
h_{\nu}(f)+\cF_*(\nu,\Psi)-5\gamma
\]
for any preassigned $\gamma>0$. We will divide the proof in several lemmas.

\begin{lemma}\label{b11}
For every $\delta>0$ and $\gamma>0$, there exists $\mu\in
\mathcal{M}_f$ such that $\mu=\sum_{i=1}^{k}a_i\mu_i$, where
$\sum_{i=1}^{k}a_i=1$ and $\mu_i\in \mathcal{E}_f$, satisfying
that
\begin{enumerate}
\item[(i)] $\mu_i(\Upsilon)=1$, $i=1,\cdots, k$; \item[(ii)]
$|\cF_*(\nu,\Phi)-\cF_*(\mu,\Phi)|<\delta$; \item[(iii)]
$h_{\mu}(f)+\cF_*(\mu,\Psi)\geq
h_{\nu}(f)+\cF_*(\nu,\Psi)-2\gamma$.
\end{enumerate}
\end{lemma}

\begin{proof}
Although the argument follows by a small modification of the standard
one in  \cite{You90} we will prove it here for completeness. Let $\mathrm{dist}^*$
denote the metric on the space of all Borel probability measures on $M$.
It follows by uniform continuity of the functional   $\eta\mapsto\cF_*(\eta, \Phi)$ (see \cite{FH10})
that for any given $\delta>0$ and $\gamma>0$, there exists $\varepsilon_0>0$ such that
\[
\mathrm{dist}^*(\tau_1,\tau_2)<\varepsilon_0\Rightarrow
|\cF_*(\tau_1,\Phi)-\cF_*(\tau_2,\Phi)|<\delta~\text{and}~|\cF_*(\tau_1,\Psi)-\cF_*(\tau_2,\Psi)|<\gamma.
\]
 Let $\mathcal{P}=\{P_1, \cdots, P_2\}$ be a partition of
$\mathcal{M}_f$ with $\mathrm{diam}\mathcal{P}=\max_{1\leq i\leq
k}|P_i|<\varepsilon_0$, where $|\cdot|$ denotes the diameter of a
set. Using ergodic decomposition theorem, there is a Borel
probability measure $\pi$ on $\cM_f$ such that $\pi(\cE_f)=1$ and
$\cF_*(\nu,\mathcal{G})=\int \cF_{*}(\tau, \cG)
\mathrm{d}\pi(\tau)$ for any asymptotically additive potential
$\cG$. We refer the reader to \cite{FH10} for a proof of this
fact. Moreover, let $\mathcal{E}_f'=\{\tau\in \mathcal{E}_f:
~\tau(\Upsilon)=1\}$ then $\pi(\mathcal{E}_f')=1$.
Let $a_i=\pi(P_i)$ and pick $\mu_i\in P_i\cap \cE_f'$ satisfying
$h_{\mu_i}(f)>h_{\tau}(f)-\gamma$ for all $\tau\in P_i\cap \cE_f'$.
Then we have that
\[
a_ih_{\mu_i}(f) \geq \int_{ P_i\cap \cE_f'} [
h_{\tau}(f)-\gamma]\, \mathrm{d}\pi(\tau)=\int_{ P_i} h_{\tau}(f)
\mathrm{d}\pi(\tau)-\gamma a_i
\]
Summing over $i$ it follows that the probability measure $\mu=\sum_{i=1}^{k}a_i\mu_i$
satisfies $ h_{\mu}(f)\geq h_{\nu}(f)-\gamma$. Moreover,
$$
\cF_*(\mu,\Psi)
    =\sum_{i=1}^{k}a_i\cF_*(\mu_i,\Psi)
    \geq \sum_{i=1}^{k} \int_{P_i} (\cF_*(\tau, \Psi)-\gamma) \mathrm{d}\pi(\tau)
    =\cF_*(\nu,\Psi)-\gamma.
$$
Hence, we have that $h_{\mu}(f)+\cF_*(\mu,\Psi)\geq
h_{\nu}(f)+\cF_*(\nu,\Psi)-2\gamma$. A similar argument shows that
$|\cF_*(\nu,\Phi)-\cF_*(\mu,\Phi)|<\delta$. This completes the
proof of the lemma.
\end{proof}

Fix a small $\xi>0$,  let $\psi_\xi$ and $\varphi_\xi$ be the
continuous functions given as in \eqref{eq.asymptotic}
approximating the sequences $\Psi$  and $\Phi$ respectively. Fix
$\delta=(\cF_*(\nu,\Phi)-c)/4>0$ and a small $\gamma>0$.   Using
the uniformly continuity of $\psi_\xi$ and $\varphi_\xi$, Birkhoff
ergodic theorem and Proposition~\ref{thm:BK} we can choose
$\varepsilon>0$ sufficiently small and $N$ sufficiently large so that
the following properties hold:
\begin{itemize}
\item[(i)]
 $d(x,y)<\varepsilon\Rightarrow
|\psi_\xi(x)-\psi_\xi(y)|<\gamma~\text{and}~|\varphi_\xi(x)-\varphi_\xi(y)|<\delta$;
\item[(ii)]
For each $n\geq N$ and $1\leq i\leq k$, there exist at least
$e^{[a_i  n](h_{\mu_i}(f)-\gamma)}$ points $x_{1}^{(i)}, \cdots,
x_{n_i}^{(i)}$ that are $(2g;[a_in],4\varepsilon)$-separated
with the property that: for each $1\leq j\leq n_i$, we have
    \begin{itemize}
    \item[(a)]
     $S_{[a_in]}\psi_\xi(x_{j}^{(i)})\leq [a_in](\int \psi_\xi \mathrm{d}\mu_i+\gamma)$;
    \item[(b)]
    $S_{[a_in]}\varphi_\xi(x_{j}^{(i)})\geq [a_in](\int \varphi_\xi \mathrm{d}\mu_i-2\delta)$.
    \end{itemize}
\end{itemize}
Note that $n_i\geq e^{[a_i n] (h_{\mu_i}(f)-\gamma)}$.
Since $f$ has the $g-$almost specification property, for each $k-$tuple $(j_1,\cdots, j_k)$
with $1\leq j_i\leq n_i$ we can choose a point
$$
y=y_{j_1\cdots j_k}
    \in \bigcap_{i=1}^{k}f^{-\sum_{j=0}^{i-1}[a_in]}(B_{[a_in]}(g;x_{j_i},\varepsilon)).
$$
Let $F=\{y_{j_1\cdots j_k}: (j_1,\cdots, j_k),\,1\leq j_i\leq n_i\}$ and $\hat{n}=\sum_{i}[a_in]$.
Now we prove:

\begin{lemma}\label{b12}
The set $F$ is $(\hat{n},2\varepsilon)$-separated, i.e., $B_{\hat{n}}(y,\varepsilon)\cap
B_{\hat{n}}(y',\varepsilon)=\emptyset$ for any disctinct $y,y'\in F$. In particular
$\# F \geq \exp(\hat n(h_\mu(f)-\gamma))$.
\end{lemma}

\begin{proof}
Given $y\neq y'\in F$, take distinct $k$-tuples $(j_1,\cdots, j_k)\neq (j_1',\cdots, j_k')$ such that
$y=y_{j_1\cdots j_k}$ and $y'=y_{j_1'\cdots j_k'}$.  We assume without loss of generality that
$j_1\neq j_1'$. Consequently,
$
y\in B_{[a_1n]}(g;x_{j_1},\varepsilon)~\text{and}~y'\in
B_{[a_1n]}(g;x_{j_1'},\varepsilon).
$
Therefore, there exists $\Lambda_i\in I(g;[a_1n],\varepsilon)$ for
$i=1,2$ such that $d_{\Lambda_1}(x_{j_1},y)<\varepsilon$ and
$d_{\Lambda_2}(x_{j_1'},y')<\varepsilon$. If $\Lambda=\Lambda_1\cap \Lambda_2$ then
$[a_1n] \geq \# \Lambda\geq [a_1n]-2g([a_1n],\varepsilon])$ and so
\begin{eqnarray*}
d_{\hat{n}}(y,y')&\geq& d_\Lambda (y,y')\geq
d_{\Lambda}(x_{j_1},x_{j_1'})-d_{\Lambda}(y,x_{j_1})-d_{\Lambda}(y',x_{j_1}')\\
&>&4\varepsilon-\varepsilon-\varepsilon=2\varepsilon.
\end{eqnarray*}
Hence $B_{\hat{n}}(y,\varepsilon)\cap
B_{\hat{n}}(y',\varepsilon)=\emptyset$  as claimed.
By the previous reasoning all $y=y_{j_1\cdots j_k}$
are distinct and so the cardinality of $F$ is bounded from below by
$n_1\cdots n_k\geq \exp (\hat n(h_{\mu}(f)-\gamma)) $. This finishes the proof of the lemma.
\end{proof}

We proceed with the following auxiliary lemma.

\begin{lemma}\label{b13}
If $\xi,\de>0$ are small enough then for every large $n$ and $y\in F$:
\begin{itemize}
\item[(i)] $ B_{\hat{n}}(y,\varepsilon)\subset U_{\hat{n}} $;
\item[(ii)] $\frac{1}{\hat{n}} \psi_{\hat{n}}(y)\geq
\cF_*(\mu,\Psi )-2\gamma$.
\end{itemize}
\end{lemma}

\begin{proof}
Fix a small $\xi>0$ and let $\varphi_\xi$ be a continuous function approximating the sequence  $\Phi$
as before. For each $z\in B_{\hat{n}}(y,\varepsilon)$ with $y=y_{j_1\cdots j_k}$
consider the images $y_i=f^{([a_1n]+\cdots +[a_{i-1}n ])}(y)$ $(1 \le i \le k)$ and notice that
\begin{align*}
\frac{1}{\hat{n}} \varphi_{\hat{n}}(z)&\geq \frac{1}{\hat{n}} S_{\hat{n}}\varphi_\xi (z)-\xi 
    \geq\frac{1}{\hat{n}}\sum_{i=1}^{k} \left( S_{[a_in]}\varphi_\xi(y_i)-[a_i n] \delta\right) -\xi\\
    &\geq\frac{1}{\hat{n}}\sum_{i=1}^{k}\Big[S_{[a_in]}\varphi_\xi(x_{j_i})- 2 [a_in]\delta
    -2g([a_in],\varepsilon)||\varphi_\xi|| \Big ]-\xi\\
    &\geq \frac{1}{\hat{n}}\sum_{i=1}^{k}
    \Big [ [a_in]\left(\int \varphi_\xi \mathrm{d}\mu_i-2\delta\right)-2g([a_in],\varepsilon)||\varphi_\xi||
    \Big ]-2\delta-\xi\\
    &\geq \frac{1}{\hat{n}}\sum_{i=1}^{k} \Big [ [a_in] \cF_*(\mu_i,\Phi)-2g([a_in],\varepsilon)||\varphi_\xi|| \Big
    ]-4\delta-2\xi\\
    &\geq\cF_*(\nu,\Phi)-5\delta-2\xi
\end{align*}
for all sufficiently large $n$. Therefore, $\frac{1}{\hat{n}} \varphi_{\hat{n}}(z)>c$ provided that $\delta$
and $\xi$ are sufficiently small. This completes the proof of (i) above. Since the arguments to prove (ii)
are analogous we shall omit the proof.
\end{proof}

We are now in a position to finish the proof of Theorem~\ref{deviations}(4). Indeed,
combining the previous lemmas we obtain that
\begin{align*}
\liminf_{\hat{n}\to \infty}  \frac{1}{\hat{n}} \log m(U_{\hat{n}})
    &\geq \liminf_{\hat{n}\to \infty}\frac{1}{\hat{n}} \log \sum_{y\in F} m(B_{\hat{n}}(y,\varepsilon))\\
    & \geq \liminf_{\hat{n}\to \infty}\frac{1}{\hat{n}} \log \sum_{y\in F} C_{\hat{n}}\exp(-\hat{n}K+\psi_{\hat{n}}(y))\\
    & \geq -K+h_{\mu}(f)+\cF_*(\mu,\Psi)-3\gamma \\
    & \geq -K+h_{\nu}(f)+\cF_*(\nu,\Psi)-5\gamma,
\end{align*}
which proves the last statement and finishes the proof of the theorem.
\end{proof}

\subsection{Proof of Theorem B}\label{s.proof2}

Here we prove our second main result, that estimates the measure of the deviation sets in terms of some thermodynamical quantities for asymptotically additive  and sub-additive families of potentials.

\vspace{.4cm} \noindent {\bf Case I.} \emph{$\Psi=\{\psi_n\}_n$
asymptotically  additive sequence of continuous functions}

\subsubsection*{Upper bound}

Since the upper bound is similar for both asymptotically additive
and sub-additive potentials we shall focus on the first case. So,
let $\Psi=\{\psi_n\}\in\mathcal A$ be an asymptotically additive
sequence of continuous functions, $\nu$ be a weak Gibbs measure
with respect to $\Phi$ on $\Lambda\subset M$ and take $c\in\R$.
Therefore, there exists $\varepsilon_0>0$  such that for every
$0<\vep<\varepsilon_0$ there exists a sequence of positive
constants $(K_n)_{n\geq1}$ (depending only on $\varepsilon$)
satisfying $\lim\limits_{n\to\infty} \frac{1}{n} \log K_n=0$ and
\begin{equation*}
K_{n}^{-1}
        \leq \frac{\nu(B_n(x,\vep))}{e^{-n P(f,\Phi)+\varphi_n(x)}} \leq
K_n,~\forall n\geq 1 \text{ and } x\in \Lambda.
\end{equation*}
In consequence $\Phi\in \mathcal{V}_K^+(g)$ with respect to the reference
measure $\nu$, where we consider the mistake function $g\equiv 0$, $C_n=K_n$ and $K=P(f,\Phi)$.  In this way,
(UB) is a direct consequence of Theorem A (2).
This finishes the proof of the upper bound in Theorem~\ref{thm.deviations}.

\subsubsection*{Lower bound using ergodic measures}

The same reasoning as above yields that $\Phi\in
\mathcal{V}_K^-(g)$ with respect to the reference measure $\nu$,
where we consider the mistake function $g\equiv 0$,
$C_n=K_{n}^{-1}$ and $K=P(f,\Phi)$. In this way, the lower bound
estimate using ergodic measures is a direct consequence of Theorem
A (3). However, for completeness we shall prove this fact to
collect some constants needed to the proof of (LB).
Let $\Psi=\{\psi_n\}$ be an asymptotically additive sequence of
continuous functions, $c\in\R$ be fixed and $\beta>0$ be
arbitrary. Denote by $U_n$ the set of points $x\in M$ so that
$\psi_n(x) > c n$, without loss of generality, we assume that
$U_n\subset \Lambda$. We claim that if $\eta$ is any $f$-invariant
and ergodic probability measure so that $\eta(\Lambda)=1$ and
$\cF_*(\eta,\Psi)>c$ then
\begin{align*}
 \liminf_{n \to \infty} & \frac{1}{n}\log  \nu(U_n)
    \geq -P(f,\Phi)+h_\eta(f) + \cF_*(\eta,\Phi)-2\beta
\end{align*}
for any preassigned $\beta>0$ with $\cF_*(\eta, \Psi)>c+\beta$.
By ergodicity one can pick $N\ge 1$ large and $\vep_0$ small so
that $N\left(n,2\vep,\frac12\right) \geq e^{[h_\eta(f)- \beta]n}$
 for all $0<\vep<\vep_0$ and $n\ge N$, and also the set $F$ of
points $x\in M$ satisfying $\cF_*(\eta,\Phi)-\beta<
\frac{1}{n}\varphi_n(x) < \cF_*(\eta,\Phi)+\beta$ and
$\cF_*(\eta,\Psi)-\beta< \frac{1}{n}\psi_n(x) <
\cF_*(\eta,\Psi)+\beta$ for all $n \ge N$ has $\eta$-measure at
least $\frac12$.

Now, take  $\xi=(\cF_*(\eta,\Psi)-\beta-c)/4>0$ and let $\psi_\xi$
be given by approximation \eqref{eq.asymptotic}. Up to consider
smaller $\vep_0$ and larger $N\ge 1$ we may assume  that
$|\psi_\xi(x)-\psi_\xi(y)|<\xi$ whenever $d(x,y)<\vep_0$
($\psi_\xi$ is uniformly continuous) and $\|\psi_n-S_n
\psi_\xi\|\leq \xi n$ for all $n\ge N$.
Using our construction it is clear that $F\subset U_n$. We claim
that
\begin{equation}\label{eq:mcontrol}
B_n(x,\vep) \subset U_n
    \quad\text{ for all } x \in F.
\end{equation}
In fact, if $x\in F$ and $y\in B_n(x,\vep)$ using $\frac1n
\psi_n(x) \geq \cF_*(\eta,\Psi)-\beta=c+4\xi$ we get
$$
|\psi_n(x)-\psi_n(y)|
    \le |\psi_n(x)-S_n\psi_\xi(x)|+|\psi_n(y)-S_n\psi_\xi(y)|+ |S_n\psi_\xi(x)-S_n\psi_\xi(y)|
    <3\xi n
$$
and consequently $\psi_n(y)>cn$, proving our claim. Therefore
$
U_n  \supset \bigcup_{x \in  F} B_n(x,\vep)
        \supset  F
$
for every $0<\vep<\vep_0$ and $n\ge N$. Moreover, if $F_n \subset
F$ is a maximal $(n,2\vep)$-separated set, one can use that the
dynamical balls $B_n(x,\vep)$ centered at points in $F_n$ are
pairwise disjoint and contained in $U_n$. This yields
\begin{align*}
\nu(U_n)
    & \geq \nu\Big( \bigcup_{x\in F_n} B_n(x,\vep) \Big)
    = \sum_{x\in F_n} \nu\Big( B_n(x,\vep) \Big) \\
     & \geq \sum_{x\in F_n} K_{n}^{-1} e^{-P(f,\Phi) n
     +\varphi_n(x)}\\
    &\geq K_{n}^{-1}\exp(-P(f,\Phi) +h_\eta(f) +\cF_*(\eta,\Phi)-2\beta)n
\end{align*}
whenever $n \geq N$, which proves our claim since the sequence
$\{K_n\}$ is tempered. The arbitrariness of  $\beta$ and the measure implies that
\begin{align*}
 \liminf_{n \to \infty} & \frac{1}{n}\log  \nu\bigg\{x \in M : \frac{1}{n} \psi_n(x) >
 c\bigg\}
    \geq -P(f,\Phi)+h_\eta(f) + \cF_*(\eta,\Psi)
\end{align*}
for every $f$-invariant, ergodic probability measure $\eta$ such that  $\cF_*(\eta,\Psi)>c$ and $\eta(\Lambda)=1$.
In consequence we obtain the second assertion in Theorem~\ref{thm.deviations}.

\vspace{.1cm}
\subsubsection*{Lower bound over all invariant measures}

Here we deduce the general lower bound over all invariant measures under the assumption that
$f$ satisfies some weak specification properties.

Fix $c>0$ and an asymptotically additive potential
$\Psi=\{\psi_n\}$.  Let $U_n=\{x\in M: \frac 1 n \psi_n(x)>c\}$.
Without loss of generality, we assume that $U_n\subset \Lambda$.
We claim that for any $f$-invariant probability measure $\eta$
satisfying $\eta(\La)=1$ and $\cF_*(\eta,\Psi) > c$ it holds that
\begin{align*}
 \liminf_{n \to \infty} & \frac{1}{n}\log  \nu\bigg[x \in M : \frac{1}{n} \psi_n(x) > c\bigg]
     \geq -P(f,\Phi)+h_\eta(f) + \cF_*(\eta,\Phi).
\end{align*}
We will make use of the following approximation result that is a reformulation of Lemma~\ref{b11}
and in a position to finish the proof of the theorem for asymptotically additive sequences.

\begin{lemma}\label{l.ergodic.approximation}
Let $\eta=\int \eta_x d\eta(x)$ be the ergodic decomposition of an
$f$-invariant probability measure $\eta$ such that $\eta(\La)=1$.
Given $\gamma>0$ and a finite set of asymptotically additive
potentials $(\Psi_j)_{1 \leq j\leq r}$, there are positive real
numbers $(a_i)_{1 \leq i \leq k}$ satisfying $a_i\leq 1$ and $\sum
a_i=1$, and finitely many points $x_1, \dots, x_k$ such that the
ergodic measures $\eta_i=\eta_{x_i}$ from the ergodic
decomposition satisfy
\begin{itemize}
\item[(i)] $\eta_i(\La)=1$; \item[(ii)] $h_{\hat\eta}(f) \geq
h_{\eta}(f) - \gamma$; and \item[(iii)] $|\cF_*(\Psi_j,\hat\eta) -
\cF_*(\Psi_j ,\eta) |
            <\gamma$ for every $1 \leq j \leq r$;
\end{itemize}
where $\hat \eta= \sum\limits_{i=1}^k a_i \eta_i$.
\end{lemma}

Given a small $\gamma>0$, using the lemma, there are ergodic
probability measures $(\eta_i)_i$ so that a probability measure
$\hat \eta= \sum_{i=1}^k a_i \eta_i$ satisfies $h_{\hat\eta}(f)
\geq h_{\eta}(f) - \gamma$,  and such that $|\cF_*(\Psi, \hat
\eta) -\cF_*(\Psi, \eta)|<\gamma$ and $|\cF_*(\Phi, \hat \eta) -
\cF_*(\Phi, \eta)|<\gamma.$
Proceeding as before there are $N\ge 1$ large and $\vep_0$ small
so that for every $1\leq i\leq k$,
the set $F^i$ of points $x\in \Lambda$ such that
$
\cF_*(\eta_i,\Phi)-\gamma<\frac{1}{n}\varphi_n(x) < \cF_*(\eta_i,\Phi)+\gamma,
$
and
$
\cF_*(\eta_i,\Psi)-\gamma<\frac{1}{n}\psi_n(x)  < \cF_*(\eta_i,\Psi)+\gamma
$
for every  $n\ge N$ and $0<\vep<\vep_0$ has $\eta_i$-measure at
least $\frac12$. We may assume also without loss of generality
that the minimum number of $(n,4\vep)-$dynamical balls needed to
cover a set of $\eta_i$-measure $\frac12$ satisfies
$N_i\left(n,4\vep,\frac12\right) \geq e^{[h_{\eta_i}(f)-
\gamma]n}$ for all $0<\vep<\vep_0$ and $n\ge N$. So, pick a finite
set $F_{n}^i \subset F^i$ so that $F_n^i$ is a maximal $([a_i n],4\vep)$-separated set in
$F^i$ and $\# F_{n}^i \geq e^{ (h_{\eta_i}(f)-\gamma) \,[a_i n]}$.
Moreover, by the specification property, for every sequence
$(x_1,x_2,\dots,x_k)$ with $x_i \in F_n^i$ there exists $x\in M$
that $\vep$-shadows each $x_i$ during $[a_i n]$ iterates with a
time lag of $N(\vep, n)$ iterates in between. Set $\ti n= \sum_{i}
[a_i n] + k N(\vep,n)$ and consider the set
$$
\mathcal{L}_{\tilde n}
    = \bigcup \{ B_{\tilde n}(x,\vep) : x=x_{i_1,\dots, i_k} \}.
$$
Similar arguments as in the proof of Lemma \ref{b12} and
Lemma \ref{b13}, we can show that the dynamical balls in the later
union have the following properties:
\begin{enumerate}
\item[i.] the dynamical balls $B_{\tilde n}(x,\vep)$ and $B_{\tilde n}(y,\vep)$ are disjoint for $x\neq y$;
\item[ii.] the number of the dynamical balls is larger than  $\exp \tilde n
(h_{\hat{\eta}}(f)-\gamma)$;
\item[iii.] each dynamical ball $B_{\tilde n}(x,\vep)$ is contained in $U_{\tilde n}$;
\item[iv.]  $\frac{ 1}{ \tilde n} \varphi_{\tilde n }(y) \geq \cF_*(\hat \eta, \Phi)-2\gamma$
for every $y\in \mathcal{L}_{\tilde n}$.
\end{enumerate}
Since $\nu$ is a weak Gibbs measure for $\Phi$ on $\Lambda\subset
M$ and $\tilde n\geq N$ it follows that  $ \nu( B_{\tilde n}(x,\vep)) \geq K_{
\tilde n}^{-1} e^{-\tilde n P(f,\Phi)+\varphi_{\tilde n}(x)}$.
Therefore, it yields that
\begin{align*}
\nu(U_{\ti n})
        &\geq  \sum_{x} \nu(B_{\ti n}(x,\vep))
        \geq K_{\tilde n}^{-1} \exp [-P(f,\Phi)+ h_{\hat\eta}(f)+\cF_*(\hat\eta, \Phi) -3\gamma] \ti n \\
    &\geq  K_{\tilde n}^{-1} \exp [-P(f,\Phi)+ h_{\eta}(f)+\cF_*(\eta, \Phi) -5\gamma] \ti n
\end{align*}
for every large $n$. Since the sequence $\{K_n\}$ is tempered and
$\gamma$ was chosen arbitrarily then we get
$$
 \liminf_{n \to \infty}
    \frac{1}{n}\log \nu(U_n )
    \geq -P(f,\Phi)+ h_{\eta}(f)+\cF_*(\eta,
\Phi)
$$
which finishes the proof of Theorem~\ref{thm.deviations} in the
case of asymptotically additive sequences $\Psi$.

\vspace{.4cm} \noindent {\bf Case II.} \emph{$\Psi=\{\psi_n\}_n$
sub-additive sequence of continuous potentials}
\vspace{.3cm}

Through the remaining of the proof, assume that $\Psi=\{\psi_n\}_n$ is a sub-additive
sequence of continuous potentials that satisfy the weak Bowen condition, that
$(\frac{\psi_n}{n})_n$ is equicontinuous and that $\inf_{n\ge 1} \frac{\psi_n(x)}n>-\infty$ for every $x\in M$
as in Theorem~\ref{thm.deviations}.
Since the proof of upper and lower bounds are similar to the previous ones for
asymptotically additive potentials we will only highlight the main differences.

\subsubsection*{Upper bound}

Let $\Psi=\{\psi_n\}\in\mathcal A$ be as above, $\nu$ be a weak Gibbs measure
with respect to $\Phi$ on $\Lambda\subset M$ and take $c\in\R$.
Therefore, there exists $\varepsilon_0>0$  such that for every
$0<\vep<\varepsilon_0$ there exists a sequence of positive
constants $(K_n)_{n\geq1}$ (depending only on $\varepsilon$)
satisfying $\lim\limits_{n\to\infty} \frac{1}{n} \log K_n=0$ and
\begin{equation*}
K_{n}^{-1}
        \leq \frac{\nu(B_n(x,\vep))}{e^{-n P(f,\Phi)+\varphi_n(x)}} \leq
K_n,~\forall n\geq 1 \text{ and } x\in \Lambda.
\end{equation*}
In consequence $\Phi\in \mathcal{V}_K^+(g)$ with respect to the reference
measure $\nu$, where we consider the mistake function $g\equiv 0$, $C_n=K_n$ and $K=P(f,\Phi)$.
On the other hand, using that $\{\frac{\psi_n}{n}\}$ is equicontinuous, $\psi_n(x)
\leq n \|\psi_1\|$ and that $\inf_{n\ge 1}
\frac{\psi_n(x)}n>-\infty$ for every $x\in M$ it follows from
Arz\'ela-Ascoli theorem that the sequence $\{\frac{\psi_n}n\}$
admits a subsequence uniformly convergent to some continuous
function $g$.
Therefore, for any $\vep>0$ small there exists a positive integer
$k\ge 1$ such that $\|\frac{\psi_k}{k}-g\|<\vep$.
For each sufficiently large $n$, it follows from \cite[Lemma
2.2]{chz} that there exists a constant $C$ depending only on $k$
such that
$
\psi_n\leq S_n(\frac{\psi_{k}}{k})+C.
$
Consequently, we have that
$$
\frac{\psi_n(x)}{n}
    \leq \frac 1nS_n\Big(\frac{\psi_{k}}{k}\Big)(x)+\frac Cn\leq \frac 1 n S_n g(x) +2\vep
$$
for all $x$ and sufficiently large $n$.
Thus, for any $\vep>0$ one has the inclusion
$$
\Big\{ \frac1n \psi_n \ge c \Big\}
    \subset \Big\{ \frac1n S_n g \ge c-2\vep \Big\}
$$
provided that $n$ is large.
Now we proceed as in the proof of Theorem A~(2). Fix $\vep>0$ be
arbitrary small and take a maximal  $(n,\varepsilon)$-separated
set  $E_n$  contained in $\{S_n g \ge (c-2\vep) n \}$. Then, the
same arguments as before show that any weak$^*$ limit $\mu$ of the
probability measures
$$
\mu_n=\sum_{x\in E_{n} }\frac{e^{\varphi_n(x)}}{Z_n}\cdot \frac 1
n \sum_{i=0}^{n-1}\delta_{f^ix}
$$
where $Z_n=\sum_{x\in E_n}e^{\varphi_n(x)} $, satisfies (UB).

It remains to prove that $\cF_*(\mu,\Psi)\geq c$. Applying
sub-additive ergodic  theorem to the invariant measure $\mu$,
there exist a $f$-invariant function $\tilde \psi$ and a subset
$\Upsilon$ with $\mu(\Upsilon)=1$ such that
$$
\tilde \psi(x)
        = \lim_{n\to\infty} \frac{\psi_n(x)}n,~~\forall x \in
        \Upsilon.
$$
This implies that $\tilde \psi(x)=g(x)$ for each point $x\in
\Upsilon$, and consequently,
$$
\cF_*(\mu,\Psi)=\int \tilde \psi \, \mathrm{d}\mu=\int g\,
\mathrm{d}\mu.
$$
Since the same argument holds for any $f$-invariant probability measure
this also implies that the map $\mu\mapsto \cF_*(\mu,\Psi)$ is
continuous since the function $g$ is continuous. Proceed as in the
proof of theorem A(2), we have
$$\int g \, \mathrm{d}\mu
    = \lim_{n\to\infty} \int g \, \mathrm{d}\mu_n
    \geq c-2\vep.
$$
The arbitrariness of $\vep$ implies that $\cF_*(\mu,\Psi)\geq c$.
This finishes the proof of the upper bound in this sub-additive
setting.

\subsubsection*{Lower bounds}

The proof of this second case goes along the same lines of the
first one. For that reason we shall just highlight the
differences.
Fix $c>0$ and let $\Psi=\{\psi_n\}$ be a sub-additive family of continuous potentials as above, and
assume without loss that  $U_n=\{x\in M: \psi_n(x)>c n\} \subset \Lambda$.
We claim that
\begin{align*}
 \liminf_{n \to \infty} & \frac{1}{n}\log  \nu\bigg[x \in M : \frac{1}{n} \psi_n(x) > c\bigg]
     \geq -P(f,\Phi)+h_\eta(f) + \cF_*(\eta,\Phi).
\end{align*}
for any $f$-invariant probability measure $\eta$
so that $\eta(\La)=1$ and $\cF_*(\eta,\Psi) > c$.
In fact, the need of the former asymptotically additive assumption used to deduce the claim
for ergodic measures was only to control the variation of this sequence on dynamical balls. See
\eqref{eq:mcontrol} for the precise statement. To overlap this
difficulty in this setting we use the weak Bowen property as follows.

\begin{lemma}
Let $\Psi=\{\psi_n\}_n$ be a sub-additive potentials with the weak
Bowen property and $a\in\mathbb R$. If $\psi_n(x)>a n$ then for
all $\gamma>0$ there exists $N\ge 1$ large so that all $y\in
B_n(x,\vep)$ satisfy $\psi_n(y)>(a-\gamma) n$
\end{lemma}

\begin{proof}
Let $a\in\mathbb R$ and $\gamma>0$ be arbitrary. In fact, using
that $\var(\psi_n\mid_{B_n(x,\vep)})\leq a_n$ for some sequence
$a_n/n\to 0$ it is immediate that $|\psi_n(x)-\psi_n(y)| \le a_n$
and consequently
$$
\frac1n \psi_n(y) \geq \frac1n \psi_n(x) - \frac{a_n}{n}
        \geq a-\gamma
$$
for every $y\in B_n(x,\vep)$, provided that $n$ is large. This proves the lemma.
\end{proof}

Hence the remaining of the argument to prove the claim in the case that $\eta$ is ergodic
is analogous to the one of Case I.
In the case that $\eta$ is not ergodic an extra approximation argument by ergodic measures
given by Lemma~\ref{l.ergodic.approximation} is necessary. The proof of that lemma uses the
continuity of functional $\mu \to \cF_*(\Psi,\mu)$. This is not necessarily continuous for general
sub-additive sequences but in our context $\cF_*(\Psi,\mu)=\int g \, d\mu$ varies continuously
with $\mu$. Since the remaining argument in the proof follows the one in Case I we shall omit the details.
This finishes the proof of Theorem~\ref{thm.deviations}.

\section*{Appendix A: Estimates for measures of mistake dynamical balls}

In this appendix we provide an estimate for the measure of mistake dynamical balls with respect to
the Gibbs measures obtained through additive thermodynamical formalism for uniformly expanding
transformations. In this setting there is a unique equilibrium state for $f$ with respect to any H\"older
continuous potential $\phi$, and it is equivalent to the unique Gibbs measure for $f$ with respect to $\phi$.

\begin{mainproposition}\label{prop:measureballs}
Let $X$ be a compact manifold, $f: X\to X$ be a uniformly expanding map, $\phi$ be a H\"older
continuous potential and $\mu=\mu_\phi$ be the unique equilibrium state for $f$ with respect to $\phi$.
Let $g$ be any mistake function. There exists $C>0$ and for any $\xi>0$ there exists a measurable set
$X_\xi\subset X$ such that $\mu(X_\xi)\geq 1-\xi$ and
$$
e^{-Pn +S_n\phi(x)}
    \leq \mu(B_n(x,\vep))
    \leq\mu(B_n(g;x,\vep))
    \leq C e^{2\xi n-Pn +S_n\phi(x)}
$$
for all $x\in X_\xi$, $n\geq 1$ and small $\vep$, where $P=P_{\text{top}}(f,\phi)$ denotes the topological
pressure of $f$ with respect to $\Phi$.
\end{mainproposition}

\begin{proof}
Let $\xi>0$ be arbitrary and fixed. Since $f$ is uniformly expanding then there exists a finite Markov partition
$\cQ$. Moreover, the unique equilibrium state $\mu$ for $f$ with respect to $\phi$ verifies $\mu(\partial \cQ)=0$
and satisfies the Gibbs property: there exists $C>0$ such that
$$
\frac1C \leq \frac{\mu(\cQ^{n}(x))}{e^{-Pn+S_n\phi(x)}} \leq C
$$
for all $x\in X$ and every $n\ne 1$, where  $P$ is the topological pressure of $f$ with respect to $\phi$ and
$\cQ^{n}(x))$ is the element of the partition $\cQ^{n}=\vee_{j=0}^{n-1} f^{-j}(\cQ)$ that contains $x$.
Using Lemma~\ref{yl32}, it follows that there exists a measurable set $X_\xi\subset X$
such that  $\mu(X_\xi)\geq 1-\xi$ and for which
$B_n(g;x,\vep)\subset \cup \{Q\in \cQ^{(n)} : Q\cap B_n(g;x,\vep)
\neq \emptyset\}$ and the cardinality of such sequence is bounded
by $e^{\xi n}$ provided that $\vep$ is small. The result follows immediately.
\end{proof}

Let us mention that the previous estimate also holds in the case of non-additive thermodynamical formalism
using the same approximation argument as above.

\subsection*{Acknowledgments}
The authors are grateful to Y. Cao for valuable conversations and
to the anonymous referee for a careful reading of the manuscript
and suggestions that greatly helped to improve the presentation.
This work was initiated during a visit of the first author to the
Soochow University, whose research conditions are greatly
acknowledged. The first author was partially supported by CNPq and
FAPESB. The second author was partially supported by NSFC
(11001191) and Ph.D. Programs Foundation of Ministry of Education
of China (20103201120001).


\bibliographystyle{alpha}
\bibliography{bib}

\begin{thebibliography}{2}

\bibitem{AP06}
V.~Ara\'ujo and M.J. Pac\'ifico.
\newblock Large deviations for non-uniformly expanding maps.
\newblock {\em J. Statist. Phys.}, 125:415--457, 2006.

\bibitem{BCH10}
J.~Ban, Y.~Cao and H.~Hu.
\newblock The dimensions of non-conformal repeller and average conformal repeller.
\newblock {\em Trans. Amer. Math. Soc.}, 362:727--751, 2010.

\bibitem{ba96}  L. Barreira.
\newblock A non-additive thermodynamic formalism and applications to
 dimension theory of hyperbolic dynamical systems.
 \newblock {\em Ergodic Theory and Dynamical Systems},  16: 871--927, (1996).

\bibitem{Ba06}
L.~Barreira.
\newblock Nonadditive thermodynamic formalism: equilibrium and Gibbs measures.
\newblock {\em Disc. Contin. Dyn. Syst.}, 16: 279--305, 2006.

\bibitem{BG06}
L.~Barreira, K.~Gelfert
\newblock Multifractal analysis for Lyapunov exponents on
nonconformal repellers.
\newblock {\em Commun. Math. Phys.}, 267: 393--418, 2006.




\bibitem{Bl83}
A.~M. Blokh.
\newblock Decomposition of dynamical systems on an interval.
\newblock {\em Uspekhi Mat. Nauk}, 38 :179--180, 1983.

\bibitem{Bo71}
R.~Bowen.
\newblock Entropy for group endomorphisms and homogeneous spaces.
\newblock {\em Trans. Amer. Math. Soc.}, 153:401--414, 1971.


\bibitem{Bow74}
R.~Bowen.
\newblock Some systems with unique equilibrium states.
\newblock {\em Math. Systems Theory}, 8 : 193--202, 1974.

\bibitem{bk}
M. Brin and A. Katok.
\newblock On local entropy.
\newblock \emph{Geometric Dynamics (Rio de Janeiro) (Lecture Notes in
Mathematics vol 1007)}, Spring-Verlag, Berlin-New York, 1983.


\bibitem{Buz97}
J.~Buzzi.
\newblock Specification on the interval.
\newblock {\em Trans. Amer. Math. Soc.}, 349:2737--2754, 1997.

\bibitem{Chu11}
Y.M.~Chung.
\newblock Large deviations on Markov towers.
\newblock {\em Nonlinearity}, 24:1229--1252, 2011.

\bibitem{CFH08}
Y. Cao, D. Feng and W. Huang.
\newblock The thermodynamic formalism for sub-additive potentials.
\newblock {\em Discrete Contin. Dyn. Syst.}, 20: 639--657, 2008.

\bibitem{chz}
Y. Cao, H. Hu and Y. Zhao.
\newblock Nonadditive  Measure-theoretic Pressure and  Applications to Dimensions  of an Ergodic
Measure.
\newblock {Ergod. Th. Dynam. Syst.}, (to appear 2012).


\bibitem{czc12}
W. Cheng, Y. Zhao and Y. Cao.
\newblock Pressures for asymptotically subadditive potentials
under a mistake funciton.
\newblock {\em Discrete Contin. Dyn. Syst.}, 32(2): 487--497,
2012.

\bibitem{CTY}
V.~Climenhaga, D.~Thompson and K.~Yamamoto.
\newblock Large deviations for systems with non-uniform structure.
\newblock {\em Ann. Inst. H. Poincar\'e Anal. Non Lin\'eaire}, 15:539--579, 1998.



\bibitem{CoRi08}
H.~Comman and J.~Rivera-Letelier.
\newblock Large deviations principles for non-uniformly hyperbolic rational maps.
\newblock {\em Ann. Inst. H. Poincar\'e Anal. Non Lin\'eaire}, 15:539--579, 1998.

\bibitem{Co09}
H.~Comman.
\newblock Strengthened large deviations for rational maps and full shifts, with unified proof.
\newblock {\em Nonlinearity}, 22: 1413--1429, 2009.

\bibitem{EKW94}
A. Eizenberg, Y. Kifer and B. Weiss.
\newblock Large deviations for $\mathbb{Z}^d-$actions.
\newblock {\em Commun. Math. Phys.}, 164: 433--454, 1994.

\bibitem{FL02}
D.-J. Feng and K.-S. Lau.
\newblock The pressure function for products of non-negative matrices.
\newblock {\em Math. Res. Lett.}, 9: 363--378, 2002.

\bibitem{FK12}
D. Feng and A. K\"aenm\"aki
 \newblock Equilibrium states of the pressure function for products of matrices
\newblock {\em Discrete Cont. Dyn. Sys.}, (to appear 2012)

\bibitem{FH10}
D. Feng, W. Huang,
 \newblock Lyapunov spectrum of asymptotically sub-additive potentials.
 \newblock {\em Commun. Math. Phys.}, 297: 1--43, 2010.



\bibitem{IY11}
G.~Iommi and Y. Yayama.
\newblock Almost-additive thermodynamical formalism for countable {M}arkov shifts.
\newblock {\em Nonlinearity}, 25: 165--191, 2012.


\bibitem{Ka80}
A.~Katok.
\newblock Lyapunov exponents, entropy and periodic points of diffeomorphisms.
\newblock {\em Publ. Math. IHES}, 51:137--173, 1980.

\bibitem{Ki90}
Y.~Kifer.
\newblock Large deviations in dynamical systems and stochastic processes.
\newblock {\em Transactions of the Americal Mathematical Society},
  321(2):505--524, 1990.


\bibitem{KN91}
Y.~Kifer and S.~E. Newhouse.
\newblock A global volume lemma and applications.
\newblock {\em Israel J. Math.}, 74(2-3):209--223, 1991.


\bibitem{Mel09}
I.~Melbourne.
\newblock Large and moderate deviations for slowly mixing dynamical systems.
\newblock {\em Proc. Amer. Math. Soc.}, 137:1735--1741, 2009.

\bibitem{MN08}
I.~Melbourne and M.~Nicol.
\newblock Large deviations for nonuniformly hyperbolic systems.
\newblock {\em Trans. Amer. Math. Soc.}, 360:6661--6676, 2008.

\bibitem{MV09}
A.~M\'eson and F.~Vericat.
\newblock Estimates of large deviations in dynamical systems by a non-additive
thermodynamic formalism.
\newblock {\em Far East J. Dyn. Syst.}, 11: 1--16, 2009.

\bibitem{Mu06}
A.~Mummert.
\newblock The thermodynamic formalism for almost-additive sequences.
\newblock {\em Discrete Contin. Dyn. Syst.}, 16: 435--454, 2006.


\bibitem{ps0}
C.-E. Pfister  and W.G. Sullivan, Billingsley dimension on shift spaces, {\it Nonlinearity},
 16:661-682, 2003.

\bibitem{PS05}
C.-E.~Pfister and W. Sullivan.
\newblock Large Deviations Estimates for Dynamical Systems without the
Specification Property. Application to the Beta-Shifts.
\newblock {\em Nonlinearity}, 18, 237--261 (2005).

\bibitem{ps}
C.-E. Pfister  and W.G. Sullivan, On the topological entropy of
saturated sets, {\it Ergod. Th. Dynam. Syst.}, 27:929-956, 2007.

\bibitem{PS09}
M.~Pollicott and R.~Sharp.
\newblock Large deviations for intermittent maps.
\newblock {\em Nonlinearity}, 22: 2079--2092, 2009.

\bibitem{re}
A. R\'enyi. \newblock Representations for real numbers and their
ergodic properties. \newblock {\em Acta Math. Acad. Sci. Hung.},
8: 477--493,1957.

\bibitem{RY08}
L.~Rey-Bellet and L.-S. Young.
\newblock Large deviations in non-uniformly hyperbolic dynamical systems.
\newblock {\em Ergodic Theory Dynam. Systems}, 28(2):587--612, 2008.

\bibitem{rvz}
J. Rousseau, P. Varandas and Y. Zhao.
\newblock  Entropy formula for dynamical systems with mistakes.
\newblock {\em Disc. Cont. Dynam. Sys.}, (to appear 2012).



\bibitem{STV03}
B.~Saussol, S.~Troubetzkoy, and S.~Vaienti.
\newblock Recurrence and {L}yapunov exponents.
\newblock {\em Mosc. Math. J.}, 3:189--203, 2003.

\bibitem{SSY09}
K.~Sakai, N.~Sumi and K.~Yamamoto.
\newblock Diffeomorphisms satisfying the specification property.
\newblock {\em Proc. Amer. Math. Soc.}, 138:315--321, 2009.

\bibitem{SVY}
N.~Sumi, P. Varandas and K.~Yamamoto.
\newblock Partial hyperbolicity and specification.
\newblock Preprint 2013.



\bibitem{Th10}
D.~Thompson
\newblock Irregular sets, the beta-transformation and the almost specification property.
\newblock Trans. Amer. Math, Soc. (to appear).

\bibitem{VV10}
P.~Varandas and M.~Viana.
\newblock Existence, uniqueness and stability of equilibrium states for
  non-uniformly expanding maps.
\newblock {\em Ann. I. H. Poincar\'e - Analyse Non-Lineaire}, 27:555--593, 2010.

\bibitem{Va12}
P.~Varandas.
\newblock Non-uniform specification and large deviations for weak Gibbs measures.
\newblock {\em J. Statist. Phys.}, 146, 330--358, 2012.


\bibitem{You90}
L.-S. Young.
\newblock Some large deviations for dynamical systems.
\newblock {\em Trans. Amer. Math. Soc.}, 318:525--543, 1990.

\bibitem{Yu00}
M. Yuri.
\newblock Weak {G}ibbs measures for certain nonhyperbolic systems.
\newblock {\em Ergod. Th. and Dyn. Sys.}, 20: 1495--1518, 2000.



\bibitem{Yu07}
M. Yuri.
\newblock Large deviations for countable to one Markov systems.
\newblock {\em Comm. Math. Phys.}, 258:455--474, 2005.


\bibitem{ZZC11}
Y. Zhao, L. Zhang and Y. Cao, The asymptotically additive
topological pressure on the irregular set for asymptotically
additive potentials, Nonlinear Analysis, 74, (2011), 5015-5022.



\end{thebibliography}

\end{document}